\documentclass[10pt, a4paper]{amsart}
\usepackage{amscd,amsmath,amssymb,amsfonts,verbatim}
\usepackage[cmtip, all]{xy}

\usepackage{comment}

\usepackage[colorlinks,linkcolor=blue,citecolor=blue,urlcolor=red]{hyperref}

\usepackage{MnSymbol}

\usepackage{tikz-cd}

\newtheorem{thm}{Theorem}[section]
\newtheorem{prop}[thm]{Proposition}
\newtheorem{lem}[thm]{Lemma}
\newtheorem{cor}[thm]{Corollary}

\theoremstyle{definition}
\newtheorem{ques}[thm]{Question}

\newtheorem{defn}[thm]{Definition}
\theoremstyle{remark}
\newtheorem{remk}[thm]{Remark}
\newtheorem{remks}[thm]{Remarks}

\newtheorem{exm}[thm]{Example}
\newtheorem{exms}[thm]{Examples}
\newtheorem{notat}[thm]{Notation}
\numberwithin{equation}{section}

{\hfill$\square$\end{defn}}
\newenvironment{rem}{\begin{remk}}%
{\hfill$\square$\end{remk}}
{\hfill$\square$\end{remks}}
{\hfill$\square$\end{exm}}
{\hfill$\square$\end{exms}}
{\hfill$\square$\end{notat}}

\newcommand{\thmref}{Theorem~\ref}
\newcommand{\propref}{Proposition~\ref}

\newcommand{\lemref}{Lemma~\ref}

\newcommand{\sC}{{\mathcal C}}

\newcommand{\sF}{{\mathcal F}}

\newcommand{\sH}{{\mathcal H}}
\newcommand{\sI}{{\mathcal I}}
\newcommand{\sJ}{{\mathcal J}}
\newcommand{\sK}{{\mathcal K}}
\newcommand{\sL}{{\mathcal L}}
\newcommand{\sM}{{\mathcal M}}

\newcommand{\sO}{{\mathcal O}}

\newcommand{\sR}{{\mathcal R}}

\newcommand{\sW}{{\mathcal W}}

\newcommand{\sZ}{{\mathcal Z}}
\newcommand{\A}{{\mathbb A}}

\newcommand{\C}{{\mathbb C}}

\renewcommand{\H}{{\mathbb H}}

\renewcommand{\P}{{\mathbb P}}
\newcommand{\Q}{{\mathbb Q}}

\newcommand{\Z}{{\mathbb Z}}

\newcommand{\fp}{{\mathfrak p}}

\newcommand{\Ker}{{\rm Ker}}

\newcommand{\CH}{{\rm CH}}

\newcommand{\surj}{\twoheadrightarrow}
\newcommand{\inj}{\hookrightarrow}
\newcommand{\red}{{\rm red}}

\newcommand{\Pic}{{\rm Pic}}
\newcommand{\Div}{{\rm Div}}
\newcommand{\Hom}{{\rm Hom}}

\newcommand{\Spec}{{\rm Spec \,}}
\newcommand{\sing}{{\rm sing}}

\newcommand{\divf}{{\rm div}}

\newcommand{\id}{{\operatorname{id}}}

\newcommand{\Sch}{{\operatorname{\mathbf{Sch}}}}

\newcommand{\Sm}{{\mathbf{Sm}}}

\newcommand{\cyc}{{\operatorname{\rm cyc}}}

\newcommand{\Nis}{{\operatorname{Nis}}}
\newcommand{\et}{{\text{\'et}}}

\newcommand{\ds}{{/\kern-3pt/}}

\newcommand{\ov}{\overline}

\renewcommand{\dim}{\text{\rm dim}}

\newcommand{\tuborg}{\left\{\begin{array}{ll}}
\newcommand{\sluttuborg}{\end{array}\right.}

\newcommand{\zar}{{\rm zar}}

\newcommand{\edim}{{\rm edim}}

\newcommand{\reg}{{\rm reg}}

\newcommand{\wt}{\widetilde}

\newcommand{\etl}{{\acute{e}t}}

\def\cO{\mathcal{O}}

\def\ol#1{\overline{#1}}

\newcounter{elno}

\newcounter{elno-abc}   
\newenvironment{listabc}{
                         \begin{list}{\alph{elno-abc})
                                     }{\usecounter{elno-abc}}
                      }{
                         \end{list}}
\newcounter{elno-abc-prime}   
\newenvironment{listabcprime}{
                         \begin{list}{\alph{elno-abc-prime}')
                                     }{\usecounter{elno-abc-prime}}
                      }{
                         \end{list}}

\begin{document}
\title{Zero-cycle groups on algebraic varieties}
\author{Federico Binda and Amalendu Krishna}
\address{Dipartimento di Matematica ``Federigo Enriques'',
  Universit\`a degli Studi di Milano\\ Via Cesare Saldini 50, 20133 Milano,
  Italy}
\email{federico.binda@unimi.it}
\address{Department of Mathematics, Indian Institute of Science,  
Bangalore, 560012, India.}
\email{amalenduk@iisc.ac.in}




\begin{abstract}
  We compare various groups of 0-cycles on quasi-projective varieties over
  a field. As applications, we show that for certain singular projective
  varieties, the Levine-Weibel Chow group of 0-cycles coincides with the corresponding 
Friedlander-Voevodsky motivic cohomology. We also show that over an algebraically closed
field of positive characteristic, the Chow group of
0-cycles with modulus on a smooth projective variety with respect to a reduced
divisor coincides with the Suslin homology of the complement of the divisor.
We prove several generalizations of the finiteness theorem of Saito and Sato for
the Chow group of 0-cycles over $p$-adic fields.
We also use these results to deduce a torsion theorem for Suslin homology which extends a result of Bloch to
open varieties.
\end{abstract}
\setcounter{tocdepth}{1}

\maketitle
\tableofcontents

\section{Introduction}\label{sec:Introduction}
The goal of this work is to clarify the relation between various groups of 0-cycles
 associated to algebraic varieties and prove some new properties of these
groups. One of the questions we want to
address is whether one can isolate a class of algebraic varieties
for which one could describe the a priori non-$\A^1$-invariant motivic
cohomology groups (such as the Levine-Weibel Chow group and the Chow group
with modulus) in terms of the~---~more classical~---~$\A^1$-invariant motivic
cohomology groups
such as the ones 
defined by Friedlander, Suslin and Voevodsky. We answer this question
for a projective variety which has either normal crossing singularities or 
it is obtained as the double of a smooth variety along a reduced divisor.
The interest in the latter case originates from the fact that it 
allows us to describe the Chow group of 0-cycles with
modulus on smooth varieties with reduced modulus in terms of 
the singular homology {\`a} la Suslin--Voevodsky. 
In the following subsections, we describe our main results and their
applications.

\subsection{Cycles with modulus vs. Suslin homology}\label{sec:CM}
Let $k$ be a field and $X$ a smooth projective variety over $k$.
Let $D \subset X$ be an effective Cartier divisor and let $U = X \setminus D$.
Let $\CH_0(X|D)$ denote
the Chow group of 0-cycles with modulus (see \cite{KS} or \cite{BS}).
These groups are supposed to provide the cycle theoretic
description of the relative algebraic $K$-theory in the same way as
the classical Chow groups do for the non-relative $K$-theory: for the case of 0-cycles this relationship has now been  established in many important cases (see e.g., \cite{BKS}, \cite{GKris2}, \cite{GKris}). 
The relative $K$-theory is an important tool to study the ordinary
$K$-theory of singular varieties by embedding them inside smooth varieties.
It is therefore important to know the structure of the Chow groups with modulus
under various special conditions on the base field and the variety.

It is not hard to see that
there is a natural surjection $\lambda_{X|D}\colon \CH_0(X|D) \surj H^{sus}_0(U)$, where
the latter is the 0-th Suslin homology of $U$ (also known as the
Suslin-Voevodsky singular homology). It acts, in many ways, as a replacement of the
classical Chow group of 0-cycles for non-proper varieties over $k$ (see  \cite{SS} for
an early example of this philosophy).

Since $H^{sus}_0(U)$ is part of the well developed theory of 
motivic cohomology with compact support, there are many tools
to study and compute it. On the other hand, the  nature of the \emph{relations}
defining $\CH_0(X|D)$ in some cases 
can be used to grasp  new information about  $H^{sus}_0(U)$ using $\lambda_{X|D}$.
It is therefore important to get conditions under which  the above surjection is an
isomorphism. 
We give a complete answer to this when $k$ is algebraically closed and
a partial answer in general.

\begin{thm}\label{thm:Main-2}
  Assume that $k$ is algebraically closed and $D$ is
  reduced{\footnote{$\lambda_{X|D}$ is almost never an
isomorphism with integral coefficients if $D$ is not reduced.}}. Then the map
\[
\lambda_{X|D} \colon \CH_0(X|D) \to H^{sus}_0(U)
\]
has the following properties.
\begin{enumerate}
\item
It is an isomorphism if ${\rm char}(k) > 0$.
\item
  It is not always an isomorphism if ${\rm char}(k) = 0$.
\end{enumerate}
\end{thm}

\begin{thm}\label{thm:Main-2-fin}
  Assume that $k$ is perfect and the irreducible components of $D_\red$ are regular.
  Then the following hold.
  \begin{enumerate}
\item
  If ${\rm char}(k) = p > 0$, then
  \[
    \lambda_{X|D} \colon \CH_0(X|D)[\tfrac{1}{p}] \to H^{sus}_0(U)[\tfrac{1}{p}]
  \]
  is an isomorphism.
\item
  If ${\rm char}(k) = 0$, then
  \[
    \lambda_{X|D} \colon {\CH_0(X|D)}/m \to {H^{sus}_0(U)}/m
  \]
  is an isomorphism for every integer $m \neq 0$.
\end{enumerate}
\end{thm}

Part (2) \thmref{thm:Main-2} may  appear surprising, because one would expect
$\CH_0(X|D)$ to behave like an $\A^1$-invariant theory if $D$ is reduced.
One application of this is that it implies that
the Chow groups with modulus defined by Russell \cite{RussellANT} does
not in general coincide with the above Chow group with modulus.
It also implies that Russell's Chow groups can not satisfy Bloch's formula
in terms of the top Zariski/Nisnevich cohomology of Milnor or Quillen
$K$-sheaves.

\subsection{Bloch's torsion theorem for Suslin homology}
\label{sec:BTor}
Let $U$ be a quasi-projective variety of dimension $d \ge 0$
over an algebraically closed
field $k$ which admits a smooth compactification.
Let $\ell \neq {\rm char}(k)$ be a prime number.
If $U$ itself is complete, then Bloch's torsion theorem \cite{Bloch-torsion}
(established independently from Roitman's  theorem)
affirms the existence of an isomorphism
\[ \CH_0(U)\{\ell\} \xrightarrow{\cong} 
H^{2d-1}_{\etl}(U, {\Q_\ell}/{\Z_\ell}(d)).\]
This provides a simple description of the
prime-to-$p$ torsion in the Chow group of zero cycles in terms of an {\'e}tale
cohomology group. 
This isomorphism is very subtle because there is no obvious map
in either direction and its construction is one of the main challenges.

When $U$ is no longer complete, but admits a smooth compactification as indicated above, we can replace the Chow group $\CH_0(U)$ by $H_0^{sus}(U)$, as explained in the previous paragraph. An extension of Roitman's torsion theorem for $H_0^{sus}(U)$, providing an isomorphism between the prime-to-$p$ torsion part of $H_0^{sus}(U)$ and the corresponding torsion in the semi-abelian Albanese variety of $U$, has been proved by  
Spie{\ss} and Szamuely \cite{SS}. As a consequence of our main result, we also get the following theorem (see Remark \ref{remk:SS} for a different approach).
Let $H^*_{c, \etl}(U, \sF)$ denote the {\'e}tale cohomology of $U$ with
compact support (e.g., see \cite{Milne}).

\begin{thm}\label{thm:Main-7}
Let $U$ be a quasi-projective variety of dimension $d \ge 0$
over an algebraically closed field such that $U$ admits a smooth compactification.
Then there is an isomorphism
\[
H^{sus}_0(U)\{\ell\} \xrightarrow{\cong} 
H^{2d-1}_{c, \etl}(U, {\Q_\ell}/{\Z_\ell}(d)).
\]
\end{thm}

\begin{remk}\label{remk:p-adic-tor}
When $U$ is complete, the $p$-adic analogue of 
\thmref{thm:Main-7} can be deduced from \cite[Cor.~3.18]{GS} using 
\cite[Thm.~8.1]{GL}.
We hope to address this question in the non-complete case in a future project.
It is also an interesting question if it is possible to eliminate
the condition in \thmref{thm:Main-7} that $U$ admits a smooth compactification.
We expect this to be the case. 
\end{remk}

\subsection{Finiteness of 0-cycle groups}\label{sec:FF}
To motivate our next set of results, suppose that
$k$ is a $p$-adic field and $m \ge 1$ is an integer such that $p \nmid m$. A
conjecture of Colliot-Th{\'e}l{\`e}ne predicted that if $X$ is a smooth projective
variety over $k$, then ${\CH_0(X)}/m$ is finite.
Such a finiteness result was previously known over finite and algebraically closed
fields (without any restriction on $m$).

The conjecture of Colliot-Th{\'e}l{\`e}ne was affirmatively answered by
Saito and Sato \cite{SaitoSatoAnnals} (see also Bloch's appendix to
\cite{EsnaultWittenberg} and
\cite{EKW} for simplified proofs).
It is natural to expect that a similar finiteness result would hold even if
$X$ is not complete provided we replace
${\CH_0(X)}/m$ by ${H^{sus}_0(X)}/m$. 
In the same spirit, one would expect that ${\CH_0(X|D)}/m$ is finite for $X$
projective, if we interpret ${\CH_0(X|D)}$ as a Chow group with compact support for
$X \setminus |D|$.

We give positive answers to both questions. Furthermore, we show that the finiteness
result of Saito and Sato also holds for singular schemes.
The question of finiteness of Suslin homology with finite coefficients arose during a
conversation of authors with Colliot-Th{\'e}l{\`e}ne on his
conjecture.

For a reduced quasi-projective scheme $X$ over a field, let $\CH^{LW}_0(X)$ denote the
Levine-Weibel Chow group of $X$ (see \S~\ref{sec:set-up}).

\begin{thm}\label{thm:Main-4}
 Let $X$ be an integral quasi-projective scheme over $k$ which is regular
 in codimension one.  Let $D \subset X$ be an effective Cartier divisor.
 We then have the following.
\begin{enumerate}
  \item
 ${H^{sus}_0(X)}/m$ is finite if $X$ is smooth.
\item
  ${\CH_0(X|D)}/m$ is finite if $X$ is smooth and complete.
\item
 ${\CH^{LW}_0(X)}/m$ is finite if $X$ is complete. 
\end{enumerate}
\end{thm}

We show that the above finiteness result also holds over finite and
algebraically closed fields (see \S~\ref{sec:Finiteness}).

\vskip .3cm

\subsection{Levine-Weibel Chow group vs. motivic cohomology}
\label{sec:LWMC**}
Let $k$ be a field of exponential characteristic $p \ge 1$. Let $\Lambda$ be a commutative ring, and assume that $p$ is invertible in $\Lambda$ or that $k$ admits resolution of singularities.
Let $X$ be a reduced quasi-projective $k$-scheme of dimension $d \ge 0$.

Let $H^{2d}(X, \Lambda(d))$ denote the Friedlander-Voevodsky motivic cohomology
of $X$ (see  \S~\ref{sec:set-up} for a reminder of the definition).
It was shown in \cite{KP} that there is a canonical
map
\begin{equation}\label{eqn:LW-FS}
\lambda_X \colon \CH^{LW}_0(X) \otimes_{\Z} \Lambda \to H^{2d}(X, \Lambda(d)),
\end{equation}
where $\Lambda$ is as above. Note that, a priori, the source and the target of $\lambda_X$ are very different objects: on one side we have a group of cycles, defined by generators and relations (in particular, no topology is involved), while on the other side we have a group defined either as a Hom-group in a motivic category, or as the cdh-hypercohomology of a certain complex of sheaves.

It is well known that this map can not be an isomorphism for arbitrary
varieties: for instance, $H^{2d}(X, \Lambda(d))$ does not
see the difference between $X$ and its semi-normalization. On the other
hand, one can easily show for a curve $X$ with a cusp singularity
that $\CH^{LW}_0(X) \otimes_{\Z}\Lambda$ changes after passing to
the semi-normalization of $X$.

On the other hand, it is a result of Voevodsky that for regular $X$, both groups are isomorphic, as a (very) special case of the identification of motivic cohomology with Bloch's higher Chow groups. It is therefore reasonable to ask if there is a specific class of singular varieties or, rather, a specific class of singularities, for which $\lambda_X$ is an isomorphism (for instance, we have already observed that semi-normality is a necessary condition). 

In this paper, we prove the following positive result when
$X$ has normal crossings
singularities (see Definition~\ref{defn:ncg}).

\begin{thm}\label{thm:Main-6}
Assume that $k$ is an infinite perfect field 
and $X$ is a normal crossing projective
variety of dimension $d \ge 0$ over $k$.
Then $\lambda_X$ is an isomorphism in the following cases.
  \begin{enumerate}
  \item
    $p > 1$ and $\Lambda = \Z[\tfrac{1}{p}]$.
      \item
      $m \in k^{\times}$ and $\Lambda = {\Z}/m$. 
    \end{enumerate}
  \end{thm}

We prove a similar isomorphism when $X$ is the join of a smooth
projective variety along a reduced divisor if $k$ is algebraically closed,
see \thmref{thm:LW-MC-Main}.

Note that $H^{2d}(X, \Lambda(d))$ is not known to have a presentation in terms 
of generators and relations in general. In \cite[Thm.~7.1]{EKW},
a variant (denoted by $\CH^{EKW}_0(X)$ in this paper) 
of the Levine-Weibel Chow group
was defined for $X$ as in \thmref{thm:Main-6}
in terms of generators and relations. It was shown that 
$\CH^{EKW}_0(X) \otimes_{\Z}  \Lambda$ is isomorphic to $H^{2d}(X, \Lambda(d))$.
\thmref{thm:Main-6} implies that under its hypothesis, there is an isomorphism
\[
\CH^{EKW}_0(X) \otimes_{\Z} \Lambda \xrightarrow{\cong} \CH^{LW}_0(X) \otimes_{\Z} \Lambda.
\]
As a consequence, it follows that $\CH^{EKW}_0(X) \otimes_{\Z} \Lambda$
can be defined in terms
of only one type of relations instead of two types of relations used in
\cite{EKW}.

We end the discussion of our main results with yet another consequence.
Let $X$ be a smooth projective variety of dimension $d \ge 2$
over an infinite
perfect field and let $D \subset X$ be a simple normal crossing divisor.
Let $\iota \colon D \inj X$ be the inclusion.
By the classical moving lemma for Chow groups, one knows that $\CH_1(X)$
is the quotient of the free abelian group $\sZ_1(X)_D$ of integral 
1-dimensional cycles on $X$ which do not meet $D_\sing$ by a suitable
rational equivalence relation. We have the obvious pull-back map
$\iota^* \colon \sZ_1(X)_D \to \sZ_1(D, D_\sing)$ 
(see \S~\ref{ssec:0-cycles-singschemes}) defined by taking the
scheme theoretic intersection with $D$. But it is not clear if 
it descends to a map between the Chow groups.
As a consequence of \thmref{thm:Main-6}, we can conclude the following.

\begin{cor}\label{cor:Main-8}Let $X$ and $D$ be as above, and let $\Lambda = \Z[\tfrac{1}{p}]$ if $p>1$ or $\Lambda = {\Z}/m$ with $m \in k^{\times}$. Then 
the map $\iota^* \colon \sZ_1(X)_D \to \sZ_1(D, D_\sing)$ descends to a
homomorphism
\[
\iota^* \colon \CH_1(X) \otimes_{\Z} \Lambda \to \CH^{LW}_0(D)\otimes_{\Z} \Lambda.
\]
\end{cor}

\vskip .3cm

\subsection{Overview of proofs}\label{sec:proofs}
To prove the first part of \thmref{thm:Main-2}, the key ingredients are the Roitman
torsion theorems of \cite{Krishna-2} and \cite{SS}, and the Lefschetz hyperplane
section theorem for Serre's Albanese variety from \cite{GK-1}.
Using these, the theorem is deduced from \propref{prop:Kernel-torsion}.
The latter is a general result of independent interest and estimates the kernel
of the map from the Chow group with modulus to Suslin homology. 
We prove the second part of \thmref{thm:Main-2} by constructing a specific example.

The proof of \thmref{thm:Main-2-fin} uses three main tricks. The first is
that in order to compare Chow group with modulus with Suslin homology,
we need to replace the integral normal curves used in defining the
Suslin homology by some embedded curves. To take care of this,
we pass from $X$ to a projective space over it where the curve is automatically
embedded and then use the covariant functoriality of the Chow group with
modulus. The second trick is to perform a series of monoidal transformations on
$X$ in order to make the given integral curve transverse with respect to
the modulus divisor $D$. We achieve this using a result of 
Jannsen (see the appendix of \cite{SaitoSatoAnnals}). 
The last trick is to apply a result of Miyazaki \cite{Miyazaki}
which allows us to assume that the pull-back of the divisor under blow-ups is reduced.

We prove \thmref{thm:Main-7} by first reducing it to the case of surfaces.
We solve the latter case by translating the problem to a
torsion theorem for the Chow group of a singular surface, namely, the double of
$X$ along $D$. We solve this problem using the methods of the torsion theorems
of \cite{Krishna-2}.

We deduce part (1) of Theorem~\ref{thm:Main-4} from a general finiteness result 
(see \propref{prop:Sing-proj-fin}) that we prove in this paper.
This result applies to all projective schemes over certain fields.
Using this, the finiteness of the Suslin homology follows using the
localization sequence for motivic cohomology with compact support.
The part (2) of \thmref{thm:Main-4} follows from  
Theorem~\ref{thm:Main-2-fin} and part (1) of Theorem~\ref{thm:Main-4} after we make 
$D$ reduced using a theorem of Miyazaki \cite{Miyazaki}, and a divisor with 
normal crossings using a sequence of blow ups with smooth centers.
We deduce part (3) of \thmref{thm:Main-4} from its part (1) and
an independent result \thmref{thm:LW-Suslin} which compares the Levine-Weibel
Chow group and the Suslin homology.

The strategy to prove \thmref{thm:Main-6} is to construct a surjective map
$\CH^{EKW}_0(X) \to \CH^{LW}_0(X)$ such that its composite with
$\lambda_X$ is the one shown to be an isomorphism by Esnault, Kerz and
Wittenberg. In order to construct such a map, our key step is
\thmref{thm:Reg-mod-LW}, which in turn relies on a moving lemma from
\cite{BK}.

We recollect various cycle and motivic cohomology groups in
\S~\ref{sec:set-up}. We recall and prove a result about the
Albanese maps in \S~\ref{sec:Alb-double}. These are used in
proving \thmref{thm:Main-2}. The main result of this section is
\thmref{thm:LW-MC-Main} which proves an isomorphism between
the Levine-Weibel Chow group and motivic cohomology of a
special kind of singular variety. 
The results of \S~\ref{sec:CM} are proven in 
\S~\ref{sec:MS**}. Theorem~\ref{thm:Main-4} is shown in \S~\ref{sec:Finiteness}.
In \S~\ref{sec:Surface-00}, we prove
\thmref{thm:Main-6} for surfaces and in \S~\ref{sec:high-dim}, 
we complete the proof of the general case.

\vskip .3cm

\subsection*{Notations and conventions}
If $A$ is a Noetherian ring, we shall write $\Sch_A$ for the category
of schemes which are separated and of finite type over $\Spec(A)$.
We let $\Sm_A$ denote the full subcategory of $\Sch_A$ whose objects
are smooth schemes over $\Spec(A)$. For $X\in \Sch_A$, we shall write $X_{\rm reg}$ (resp. $X_{\rm sm}$) for the regular locus of $X$ (resp. for the smooth locus of $X\to \Spec(A)$). By a regular (or smooth) closed point of $X$ we will mean a closed point lying in $X_{\rm reg}$ (resp. in $X_{\rm sm}$). We will write $X_{\rm sing}$ for the singular locus of $X$ (i.e.,  the complement $X \setminus X_{\rm reg}$) with reduced closed subscheme structure. We let $X_{(i)}$ denote the set of points $x \in X$ such that
$\dim(\ov{\{x\}}) = i$.

We will be mostly interested in the case where $A = k$ is a perfect field. 
If $p$ is the exponential characteristic of $k$,
we shall let $\Lambda$ be the ring $\Z$ if $k$ admits resolution of
singularities, or the ring $\Z[\tfrac{1}{p}]$ otherwise (\cite[Def. 3.4]{FV}).
For any abelian group $M$, we shall write $M\{p'\}$ (resp. $M\{p\}$) for the subgroup of elements in $M$ which are torsion of order prime to $p$ (resp. a power of $p$). For a commutative ring $R$, we let
$M_R = M {\underset{\Z}\otimes} R$.

In this paper, we shall describe the setting of a subsection in the beginning
of the subsection. Further assumptions will be explicitly stated whenever they are
used in the middle of a subsection.

\section{Preliminaries: cycles groups and
  motivic cohomology}\label{sec:set-up}
We begin by fixing some notations and clarifying our terminology.
Throughout this paper, we shall encounter several different groups of 0-cycles: the groups will often share the same set of generators, but will be subject to different notions of rational equivalence. In a nutshell, our main results will be a comparison between the different notions, by means of suitable moving lemmas, and to deduce new
properties of these 0-cycle group. We shall draw some consequences.
In this preliminary section, we introduce the main players, recalling some definitions already existing in the literature. New groups of cycles, tailored on the specific geometry of the varieties we are interested in, will be defined later in the paper.

\subsection{Cycle groups on singular schemes}\label{ssec:0-cycles-singschemes}
Let $k$ be a field of exponential characteristic $p \ge 1$. Let $X\in \Sch_k$ be a reduced $k$-scheme and $Y \subset X$ a closed subscheme. For any $r\geq 0$, we write
$\sZ_r(X,Y)$ for the free abelian group on the set of integral $r$-dimensional cycles in $X$ (i.e., integral closed subschemes of dimension $r$) which do not meet $Y$.
Let $\sR_r(X)$ be the subgroup of $\sZ_r(X) := \sZ_r(X, \emptyset)$ generated by the cycles which are rationally equivalent to zero (see for example \cite[\S 1]{Fulton}). We write $\CH_r(X)$ for the quotient $\sZ_r(X)/\sR_r(X)$ and refer to it as the (Fulton's) Chow group of $r$-cycles on $X$ modulo rational equivalence.
We shall also use the notation $\CH^{d-r}(X)$ for $\CH_r(X)$ if $X$ has pure
dimension $d$.

We now recall the definition of the Levine-Weibel Chow group of 0-cycles on $X$ from   
\cite{LW} and its modified version from \cite{BK}.
Let $Y \subset X$ be a nowhere dense closed subscheme containing $X_\sing$.

\begin{defn}\label{defn:0-cycle-S-1}  
Let $C$ be a reduced $k$-scheme which is of pure dimension one (from now on, such $C$ will be simply called \emph{a curve}).  
We shall say that a pair $(C, Z)$ is \emph{a good curve  
relative to $Y$} if there exists a finite morphism $\nu\colon C \to X$  
and a  closed proper subscheme $Z \subsetneq C$ such that the following hold.  
\begin{enumerate}  
\item  
No component of $C$ is contained in $Z$.  
\item  
$\nu^{-1}(Y) \cup C_{\rm sing}\subseteq Z$.  
\item  
$\nu$ is local complete intersection at every   
point $x \in C$ such that $\nu(x) \in Y$.   
\end{enumerate}  
\end{defn}

Let $(C, Z)$ be a good curve relative to $Y$ and let   
$\{\eta_1, \cdots , \eta_r\}$ be the set of generic points of $C$.   
Let $\sO_{C,Z}$ denote the semilocal ring of $C$ at   
$S = Z \cup \{\eta_1, \cdots , \eta_r\}$.  
Let $k(C)$ denote the ring of total  
quotients of $C$ and write $\sO_{C,Z}^\times$ for the group of units in   
$\sO_{C,Z}$. Notice that $\sO_{C,Z}$ coincides with $k(C)$   
if $|Z| = \emptyset$.   
As $C$ is Cohen-Macaulay, $\sO_{C,Z}^\times$  is the subgroup of group of units
in the ring of total quotients $k(C)^\times$   
consisting of those $f\in \sO_{C,x}$ which are regular and invertible for
every $x\in Z$.

Given any $f \in \sO^{\times}_{C, Z} \inj k(C)^{\times}$, we denote by    
${\rm div}_C(f)$ (or ${\rm div}(f)$ in short)   
the divisor of zeros and poles of $f$ on $C$, which is defined as follows. If   
$C_1,\ldots, C_r$ are the irreducible components of $C$,   
and $f_i$ is the factor of $f$ in $k(C_i)$, we set   
${\rm div}(f)$ to be the $0$-cycle $\sum_{i=1}^r {\rm div}(f_i)$, where   
${\rm div}(f_i)$ is the usual   
divisor of a rational function on an integral curve in the sense of   
\cite{Fulton}. As $f$ is an invertible   
regular function on $C$ along $Z$, ${\rm div}(f)\in \sZ_0(C,Z)$.  

By definition, given any good curve $(C,Z)$ relative to $Y$, we have a   
push-forward map $\sZ_0(C,Z)\xrightarrow{\nu_{*}} \sZ_0(X,Y)$.  
We shall write $\sR_0(C|Z, X|Y)$ for the subgroup  
of $\sZ_0(X, Y)$ generated by the set   
$\{\nu_*({\rm div}(f))| f \in \sO^{\times}_{C, Z}\}$.

\begin{defn}  \label{def:LWandlciChow}
  (1) Let $\sR_0(X,Y)$ denote the subgroup of $\sZ_0(X, Y)$ generated
  by   
the image of the map $\sR_0(C|Z, X|Y) \to \sZ_0(X, Y)$, where  
$(C, Z)$ runs through all good curves relative to $Y$.  
We let $\CH^{BK}_0(X,Y) = \frac{\sZ_0(X, Y)}{\sR_0(X,Y)}$ and call it the
\emph{lci Chow group} of zero cycles of $X$.
We shall write $\sR_0(X,X_\sing)$ as $\sR_0(X)$ and $\CH^{BK}_0(X,X_\sing)$
as $\CH^{BK}_0(X)$.

(2) 
Let $\sR^{LW}_0(X)$ denote the subgroup of $\sZ_0(X)$ generated  
by the divisors of rational functions on good curves as above, where  
we further assume that the map $\nu\colon C \inj X$ is a closed immersion.
We let   $\CH^{LW}_0(X)=\frac{\sZ_0(X, X_{\rm sing})}{\sR^{LW}_0(X)}$ and call it
the \emph{Levine-Weibel Chow group} of zero cycles on $X$.
We shall refer to good curves relative to $X_\sing$ as {\sl lci curves} and
embedded good curves on $X$ as {\sl Cartier curves}.
\end{defn}
By definition, there is a canonical surjection  
$\CH^{LW}_0(X) \surj \CH^{BK}_0(X)$.

\begin{rem}The Chow group $\CH^{LW}_0(X)$ was  
discovered by Levine and Weibel \cite{LW} in an attempt to describe the  
Grothendieck group of a singular scheme in terms of algebraic cycles.  
The modified version $\CH^{BK}_0(X)$ was introduced in \cite{BK}.  
\end{rem}

Note that there is always a  map
\begin{equation}\label{eqn:lci-Fulton}
  \CH^{BK}_0(X) \to \CH_0(X)
  \end{equation}
which is far, in general, from being an isomorphism if $X_{\rm sing} \neq
\emptyset$. This is already clear when $\dim (X)=1$, see \cite{LW} or
\cite[Lem. 3.12]{BK}.

Assume now that $X$ is complete of pure dimension $d \ge 0$ and let
$X_1, \ldots, X_r$ be the irreducible components of $X$. For any
0-cycle $\alpha \in \sZ_0(X, X_\sing)$, we let
$\alpha_i$ be the restriction of $\alpha$ on $X_i$ and let
$\deg(\alpha) = (\deg(\alpha_1), \ldots , \deg(\alpha_r)) \in \Z^r$.
This defines a homomorphism
$\deg \colon  \sZ_0(X, X_\sing) \to \Z^r$. It is well known and easy to check
from the definition of $\sR^{LW}_0(X)$ that $\deg(\sR^{LW}_0(X)) = 0$.
One therefore has a degree homomorphism
$\deg \colon \CH^{LW}_0(X) \to \Z^r$. Using a projective bundle trick (e.g., see
the proof of \thmref{thm:M1}), one deduces that $\deg$ actually factors
through $\CH^{BK}_0(X)$. We shall let $\CH^{LW}_0(X)_{\deg 0}$ and
$\CH^{BK}_0(X)_{\deg  0}$ denote the kernels of these degree maps.
Note that $\deg$ does not factor through  $\CH_0(X)$ unless $X$ is integral.

\subsection{Singular homology}\label{ssec:SuslinHomologydef}
Let $k$ be a field and $U$ a smooth $k$-scheme. Recall that the Suslin-Voevodsky singular homology of $U$ (also called Suslin homology)  is defined as follows. Let $\Delta^n$ be the algebraic $n$-simplex and denote by $C_n(U)$ the free abelian group generated by those integral closed subschemes $Z$ of $U\times \Delta^n$ for which the projection $Z\to \Delta^n$ is finite and surjective. By pulling back via the structure maps of the cosimplicial scheme  $\Delta^{\bullet}_k$, one equips $C_\bullet(U)$ with the structure of a simplicial abelian group. We denote the associated chain complex (via the Dold-Kan correspondence) also by $C_\bullet(U)$. The $n$-th singular homology
group $H^{sus}_n(U)$ of $U$ is the $n$-th homology of this complex.

When $n=0$, another, explicit, description is available (and will be taken as a definition for us). For $i\colon C\inj U$ an integral curve, let $\phi\colon C^N\to C$ be its normalization and let $\ol{C}^N$ be the unique projective normal
compactification of $C^N$. We denote by $C_\infty$ the complement $\ol{C}^N \setminus C^N$, and consider it as a \emph{reduced} Cartier divisor on $\ol{C}$. Write $G(\ol{C}^N, C_\infty)$ for ${\rm Ker}(\cO_{\ol{C}^N, C_\infty}^\times \to \cO_{C_\infty}^\times)$, where $\cO_{\ol{C}^N, C_\infty}$ denotes as before the semilocal ring of $\ol{C}^N$ at the union of the closed points supporting $C_\infty$. Note that $G(\ol{C}^N, C_\infty)\subset k(C)^\times = k(\ol{C}^N)^\times$.

For $f\in G(\ol{C}^N, C_\infty)$, we can consider its divisor $\divf_{\ol{C}^N}(f)\in \sZ_0(\ol{C}^N, C_\infty)$. Since $\phi$ is finite and $i$ is a closed immersion, we can consider the composite
\[ \sZ_0(\ol{C}^N, C_\infty) = \sZ_0(C^N) \xrightarrow{\phi_*} \sZ_0(C) \xrightarrow{i_*} \sZ_0(U).  \]
We write $\sR^{S}_0(C, U)$ for the subgroup of $\sZ_0(U)$ generated by the set $\{ i_*\phi_* (\divf(f)) \,|\, f\in G(\ol{C}^N, C_\infty)\}$. Let $\sR^S_0(U)$ denote the subgroup of $\sZ_0(U)$ generated by $\sR^{S}_0(C, U)$ where $C$ runs through all integral curves in $X$.

\begin{thm}$($\cite[Thm. 5.1]{Schmidt-1}$)$\label{thm:suslinhomology}
  There is a canonical isomorphism
  \[
    \frac{\sZ_0(U)}{\sR_0^S(U)} \xrightarrow{\cong} H^{sus}_0(U).
  \]
  \end{thm}

  In the sequel, we shall use the left hand side of this isomorphism as the
  definition of $H^{sus}_0(U)$. There is a
  degree homomorphism $\deg \colon H^{sus}_0(U) \to \Z^r$, where
  $r$ is the number of connected components of $U$. This is the product of
  maps on Suslin homology induced by the projection of each connected
  component of $U$ to $\Spec(k)$. We let
  $H^{sus}_0(U)_{\deg 0}$ be the kernel of $\deg$.

\begin{remk}\label{remk:mot-homology-DM}
  The $j$-th Suslin homology group of $U$ with coefficients in $\Lambda$
  has the following motivic interpretation which will be explored further later in the
  text.
\[ H_j^{sus}(U)_{\Lambda} = \Hom_{\mathbf{DM}^{\rm eff}_{\Nis}(k, \Lambda)}(\Lambda[j], M(U)),
  \]
  where $\mathbf{DM}^{\rm eff}_{\Nis}(k, \Lambda)$ is Voevodsky's category of effective
  motives with $\Lambda$-coefficients and $M(U)$ is the motive
of $U$.  

(2) It is clear from the above description that
$H_0^{sus}(U) \cong \CH_0(U)$ if $U$ is complete.
\end{remk}

\subsection{Motivic cohomology}
\label{sec:LW-mot-coh}Let $k$ be a perfect
field of exponential characteristic $p \ge 1$.
Let $X$ be a quasi-projective scheme of dimension $d \ge 0$ over $k$. 
Recall from \cite[Lecture~16]{MVW} that given $T \in \Sch_k$ and integer
$r \ge 0$, the presheaf $z_{equi}(T,r)$ on $\Sm_k$ is defined by
letting $z_{equi}(T,r)(U)$ be the free abelian group generated by
integral closed subschemes $Z \subset U \times T$ such that $Z$ is dominant and
equidimensional of relative dimension $r$ (any fiber is either empty or all
its components have dimension $r$) over a component of $U$. 
It is known that $z_{equi}(T,r)$ is a sheaf on the big {\'e}tale site of $\Sm_k$.

Let $C_{\ast}z_{equi}(T,r)$ be the chain complex of
presheaves of abelian groups associated to
the simplicial presheaf on $\Sm_k$ given by
$C_nz_{equi}(T,r)(U) = z_{equi}(T,r)(U \times \Delta^n_k)$.
The simplicial structure on $C_{\ast}z_{equi}(T,r)$ 
is though the cosimplicial scheme  $\Delta^{\bullet}_k$, as before.
Recall the following definition of motivic cohomology of singular schemes
from \cite[Defn.~9.2]{FV}.

\begin{defn}\label{defn:MC-sing}
The motivic cohomology groups of $X$ are defined as the
hypercohomology 
$H^m(X, \Z(n)) = \H^{m-2n}_{cdh}(X,  C_{\ast}z_{equi}(\A^n_k,0)_{cdh})$. Here, $C_{\ast}z_{equi}(\A^n_k,0)_{cdh}$ is the complex obtained by applying term wise the change of site functor from $\Sm_k$ to $\Sch_k$.
\end{defn}

This definition is rather untouchable. However, if one is willing to work with $\Lambda$-coefficients (or under the assumption of resolution of singularities), motivic cohomology groups have a description in terms of Hom-groups in Voevodsky's category of motives, and even on the complexes $C_{\ast}z_{equi}(\A^n_k,0)_{cdh}$ we have a better understanding. More precisely, we have the following result (see \cite[Thm.~3.10]{KP}).
\begin{thm} There is a functorial
isomorphism \[H^m(X, \Lambda(n)) \simeq \Hom_{\mathbf{DM}(k, \Lambda)}(M(X), \Lambda(n)[m]),\] where
$\mathbf{DM}(k, \Lambda)$ is Voevodsky's non-effective category of motives for the cdh-topology (also known as the `big' category of motives) with $\Lambda$-coefficients.
\end{thm}

\begin{rem}
  When $X$ is smooth, the motivic cohomology groups are isomorphic to Bloch's higher Chow groups. More precisely, for any $m,n\in \Z$, there is a natural isomorphism $H^m(X, \Z(n)) = \CH^n(X, 2n-m)$. See 
\cite[Cor. 2]{VoevHigherChow}. 
\end{rem}

In this paper, we are specifically interested in the bi-degree $(2d,d)$. By \cite[Lem.~7.12]{KP}, there is a canonical group homomorphism
\begin{equation}\label{eqn:MC-sing-0}
\gamma_X \colon \CH_0^{BK}(X)_{\Lambda} \to H^{2d}(X, \Lambda(d)).
\end{equation}
We remark here that 
the construction of $\gamma_X$ uses \cite[Thm.~5.1]{KP} whose proof
holds verbatim
with integral coefficients when $k$ admits resolution of singularities. 
The map $\gamma_X$ has the property that for a closed point $x \in X_{\rm reg}$,
the diagram
\begin{equation}\label{eqn:MC-sing-1}
\xymatrix@C1pc{
\Lambda \ar[r]^-{\simeq} & \CH_0(k(x))_{\Lambda} \ar[r]^-{\simeq} \ar[d] & 
H^0(k(x), \Lambda(0)) \ar[d] \\
& \CH_0^{BK}(X)_{\Lambda} \ar[r]^-{\gamma_X} &  H^{2d}(X, \Lambda(d))}
\end{equation}
commutes, where the vertical arrows are the push-forward maps
(see \cite[\S~7]{KP}, and \cite[\S 5.1]{BKres} where a similar construction is performed with motivic cohomology replaced by \'etale cohomology with finite coefficients). 
We let 
\[\lambda^{LW}_X\colon \CH^{LW}_0(X)_{\Lambda} \to \CH_0^{BK}(X)_{\Lambda} \to 
H^{2d}(X, \Lambda(d))\] denote the composite map.

We let $M^c(X) \in \mathbf{DM}(k, \Lambda)$ denote the motive of $X$ with compact support.
One knows that $M^c(X) \simeq M(X)$ if $X$ is complete.
Under our assumption on $\Lambda$, it follows from \cite{VoevTriangCat} and
\cite[Chap.~5]{Kelly} that for a closed immersion $Z_1 \subset Z_2$
of schemes, there is a distinguished triangle (called the localization sequence) in $\mathbf{DM}(k, \Lambda)$:
\begin{equation}\label{eqn:loc-seq}
M^c(Z_1) \to M^c(Z_2) \to M^c(Z_2 \setminus Z_1) \to M^c(Z_1)[1].
\end{equation}

Using the duality between the motivic homology and the motivic cohomology
with compact support, one gets a functorial isomorphism
$H^{sus}_n(Z)_\Lambda \simeq H^{2d-n}_c(Z, \Lambda(d))$ if $Z$ is a smooth scheme of pure
dimension $d$, see \cite{FV} and \cite[Chap.~5]{Kelly}, where the motivic cohomology groups with compact support is defined as 
\[ H^m_c(Z, \Lambda(n)) \simeq \Hom_{\mathbf{DM}(k, \Lambda)}(M^c(Z), \Lambda(n)[m]).
\]

In what follows, we will frequently use these identifications.
We shall sometimes write $H^m(Z, \Z(n))$ as $H^{m,n}(Z)$.

\subsection{Cycles with modulus}
Let $k$ be a field. Given any integral normal curve $C$ over $k$ and an effective Cartier (or Weil) divisor $E$ on $C$, we say that a rational function $f\in k(C)^\times$ on $C$ has modulus $E$ if $f\in {\rm Ker}(\cO_{C, E}^\times \to  \cO_E^\times)$. As in \S~\ref{ssec:SuslinHomologydef}, we let $G(C, E)$ denote the subgroup of $k(C)^\times$ given by such rational functions. Note that $E$ is not necessarily reduced.

Let $X\in \Sch_k$ be an integral scheme,   and let $D$ be an effective Cartier divisor on $X$. Write as before $\sZ_0(X, D)$ for the free abelian group on the set of closed points in $X \setminus D$. If $C$ is an integral normal curve over $k$ equipped with a finite morphism $\phi_C\colon C\to D$ such that $\phi_C(C)\not\subset D$, we have a well defined group homomorphism
\[\tau_C\colon G(C, \phi_C^*(D)) \to \sZ_0(X,D), \quad f\mapsto (\phi_C)_*(\divf(f)),\] 
where $\phi_C^*(D)$ denote the pullback of $D$ to $C$. Note that $\phi_C^*(D)$ is not necessarily reduced, even when $D$ is. Recall the following definition, that originally appeared in \cite{KS}.
\begin{defn}[Kerz-Saito] We define the Chow group $\CH_0(X|D)$ of $0$-cycles of $X$ with modulus $D$ as the cokernel of the homomorphism
    \[\divf\colon \bigoplus_{\phi_C\colon C\to X}G(C, \phi_C^*D)\to \sZ_0(X,D),\]
    where the sum runs over the set of finite morphisms $\phi_C\colon C\to X$ from integral normal curves such that $\phi_C(C)\not\subset D$. 
\end{defn}

The Chow groups of zero cycles with modulus are known to be covariantly functorial (pushforward) for proper maps  $f\colon X'\to X$ with respect to Cartier divisors $D'$ on $X'$ and $D$ on $X$ such that $f^*(D)$ is defined and satisfies $f^*(D) \leq D'$ as effective Cartier divisors. See \cite[Lem. 2.4]{BS} or
\cite[Prop.~2.10]{KPv}. Similarly, it is easy to construct a pullback along finite flat maps. See again \cite[Lem. 2.4]{BS} or \cite[Prop.~2.12]{KPv}.

\begin{lem}\label{lem:mod-Suslin} Assume that $X$ is complete and 
  $U=X \setminus D$ is a smooth $k$-scheme.  Then the canonical bijection
  $\sZ_0(X,D) \xrightarrow{\cong} \sZ_0(U)$ induces a surjective homomorphism
  \begin{equation}\label{eqn:mod-Suslin}\lambda_{X|D}\colon
    \CH_0(X|D)\surj H_0^{sus}(U).\end{equation}
    \end{lem}
    \begin{proof}
      This is an easy exercise using \thmref{thm:suslinhomology}.
      \end{proof}

The map \eqref{eqn:mod-Suslin} is not an isomorphism if $D$ is not reduced. If $\dim(X)=1$ and $X$ is a complete normal curve over $k$, the group $\CH_0(X|D)$ agrees with the relative Picard group $\Pic(X,D)$ of line bundles $\mathcal{L}$ on $X$ equipped with a trivialization $\mathcal{L}_{|D} \cong \cO_D$ along $D$ (see, e.g., \cite[\S 4]{BS}), while  $H_0^{sus}(U)$ agrees with the relative Picard group $\Pic(X, D_{\rm red})$, by \cite[Thm. 7.16]{MVW}. So, the two groups do not agree when $D$ is not reduced. 
\thmref{thm:Main-2} of this paper describes the nature of
$\lambda_{X|D}$ when $D$ is reduced.

\subsection{Normal crossing varieties} Let $X$ be an integral
Noetherian scheme of finite Krull dimension $d \ge 0$.
Let $D \subset X$ be an effective Cartier divisor and
let $D_i \subset D, \ i \in I$ be the irreducible components of $D$. 
Note that $D = \emptyset$ if $d = 0$.  Recall the following definition 
(e.g., see \cite[2.4]{deJong}).
\begin{defn}\label{defn:ncs}
(1) We say that $D$ is  a \emph{simple normal crossing divisor} on $X$ 
if the local ring $\cO_{X,x}$ is regular for each point $x\in D$ and for 
each $J \subset I$ nonempty, the scheme theoretic intersection 
$D_J = \bigcap_{j\in J} D_j$ is a regular scheme which is either empty or
of pure dimension $d+1-|J|$.

(2) We say that $D$ is a \emph{normal crossing divisor} on $X$ if for every 
$x \in D$, there exists an {\'e}tale morphism $U \to X$ with $x$ in its image 
such that $D \times_X U$ a simple normal crossing divisor on $U$. 
\end{defn}

\begin{defn}\label{defn:ncg}
Let $k$ be a field and $X$ a reduced quasi-projective $k$-scheme.
Let $\{X_1, \ldots , X_n\}$ be the set of irreducible components of $X$. We 
shall say that $X$ is a normal crossing variety of dimension $d \ge 0$ 
if for every nonempty subset $J \subset [1, n]$, the scheme theoretic 
intersection $X_J := {\underset{i \in J}\bigcap} X_i$ is a smooth $k$-scheme
which is either empty or of pure dimension $d + 1 - |J|$.
\end{defn}

\subsection{The double construction}\label{sec:double-construction}
Let $k$ be any field.
Let $X$ be an integral regular quasi-projective scheme of dimension $d$ over $k$
and let $D \subset X$ be an effective Cartier divisor. Recall from 
\cite[\S~2.1]{BK} that the double (or join) of $X$ along $D$ is a quasi-projective
scheme $S(X,D) = X \amalg_D X$ so that
\begin{equation}\label{eqn:rel-et-2}
\begin{array}{c}
\xymatrix@R=1pc{
D \ar[r]^-{\iota} \ar[d]_{\iota} & X \ar[d]^{\iota_+} \\
X \ar[r]_-{\iota_-} & S(X,D)}
\end{array}
\end{equation}
is a co-Cartesian square in $\Sch_k$.
In particular, the identity map of $X$ induces a finite flat map
$\Delta_X\colon S(X,D) \to X$ such that $\Delta_X \circ \iota_\pm = {\rm id}_X$
and $\pi = \iota_+ \amalg \iota_-: X \amalg X \to S(X,D)$ 
is the normalization map.
We let $X_\pm = \iota_\pm(X) \subset S(X,D)$ denote the two irreducible
components of $S(X,D)$.  We use $S_X$ as a shorthand for $S(X, D)$
when the divisor $D$ is understood. $S_X$ is a reduced quasi-projective
scheme whose singular locus is $D \subset S_X$. 
It is projective (resp.  affine) whenever
$X$ is so. It follows from \cite[Lem.~2.2]{Krishna-2} that  
~\eqref{eqn:rel-et-2}
is a bi-Cartesian square.

We have a pullback isomorphism
$\sZ_0(S_X,D)\xrightarrow{(\iota^*_{+}, \iota^*_{-})} \sZ_0(X_+,D) \oplus 
\sZ_0(X_-,D)$, as well as pushforward inclusion maps
${p_{\pm}}_* \colon \sZ_0(X,D) \to \sZ_0(S_X,D)$ such that 
$\iota^*_{+} \circ {p_{+}}_* = \id$
and $\iota^*_{+} \circ {p_{-}}_* = 0$. The main result of \cite{BK} connects the lci Chow group of zero cycles on $S_X$ with Fulton's Chow group of $X$ and the Kerz-Saito Chow group with modulus $\CH_0(X|D)$ as follows.

If $k$ is perfect{\footnote{The perfectness condition has been removed 
recently in \cite[Thm.~1.1]{GKR}.}}, there are push-forward maps
$p_{\pm,*} \colon \CH_0(X|D) \to \CH_0^{BK}(S_X)$ and pull-back maps
$\iota^*_\pm \colon \CH_0^{BK}(S_X) \to \CH_0(X)$, where the latter 
are split by $\Delta^*_X\colon \CH_0(X) \to \CH_0^{BK}(S_X)$.
By \cite[Thm.~7.1]{BK}, there is a split short exact 
sequence 
\begin{equation}\label{eqn:doubl-ex-seq}
0 \to \CH_0(X|D) \xrightarrow{p_{+,*}} \CH_0^{BK}(S_X) \xrightarrow{\iota^*_-}
\CH_0(X) \to 0
\end{equation}
and a similar split exact sequence corresponding to the maps $(p_{-,*}, \iota^*_+)$.

\section{Chow group vs. motivic cohomology}
\label{sec:Alb-double}
The goal of this section is to prove that the composition of the map
~\eqref{eqn:MC-sing-0} with the Levine-Wiebel Chow group is
an isomorphism for $S_X$.
In combination with ~\eqref{eqn:doubl-ex-seq}, this will be the key step
in the proof of \thmref{thm:Main-2}.
We begin by recalling Serre's Albanese varieties and Albanese
maps and state \thmref{thm:Semi-abl-Alb} which we shall use in the
next section.

\subsection{The semi-abelian Albanese variety for the double}
\label{sec:Alb-double*}
We fix an algebraically closed field $k$ of any characteristic. We fix a pair $(X,D)$ consisting of an integral smooth projective $k$-scheme $X$ of dimension $d\geq 1$ and an effective Cartier divisor $D$ on it. We write $S_X$ for the double $S(X,D)$ of 
\S~\ref{sec:double-construction} and $U=X \setminus D$.

Let $\pi: \ov{S}_X= X_+ \amalg X_- \to S_X$ denote again the normalization map. 
$NS(\ov{S}_X)$ denote the Neron-Severi group of $\ov{S}_X$. Recall that
this is the quotient of the class group of $\ov{S}_X$ by the subgroup of
Weil divisors which are algebraically equivalent to zero.
Let ${\rm Div}(S_X)$ denote the free abelian group of Weil divisors on $S_X$.
Let $\Pi_{1}(S_X)$ denote the subgroup of ${\Div}(\ov{S}_X)$ generated
by the Weil divisors which are supported on $\pi^{-1}(D)$.
There is a canonical map $\iota_{S_X} \colon \Pi_{1}(S_X) \to
NS(\ov{S}_X)$.
Let $\Pi_{\rm Serre}(S_X)$ denote the kernel this map and let $\Pi(S_X)$ be the
kernel of the canonical map
\begin{equation}\label{eqn:Weil-Pic}
\Pi_{1}(S_X) \xrightarrow{(\iota_{S_X}, \pi_*)} NS(\ov{S}_X) \oplus {\rm Div}(S_X).
\end{equation}

Let $J^d(S_X)$ and $J^{d}_{\rm Serre}(S_X)$ denote the Cartier duals of
the 1-motives 
\[[\Pi(S_X) \to \Pic^0(\ov{S}_X)], \text{ and }
[\Pi_{\rm Serre}(S_X) \to \Pic^0(\ov{S}_X)],\]
respectively.
There is a surjective group homomorphism
\[\rho_{S_X}\colon \CH_0^{LW}(S_X)_{\deg 0} \surj J^d(S_X)\]
which is 
universal among regular homomorphisms from $\CH_0^{LW}(S_X)_{\deg 0}$
to semi-abelian varieties. $J^d(S_X)$ is called the {\sl semi-abelian Albanese
variety} of $S_X$. By \cite[Prop.~9.7]{BK}, there is a factorization
of $\rho_X$:
\begin{equation}\label{eqn:alb-lci}
\CH_0^{LW}(S_X)_{\deg 0} \surj \CH_0^{BK}(X)_{\deg 0} 
\stackrel{\wt{\rho}_{S_X}}{\surj} J^d(S_X)
\end{equation}
and $J^d(S_X)$ is in fact the universal regular semi-abelian variety quotient of
$\CH_0^{BK}(S_X)_{\deg 0}$. 

One knows from \cite{Serre} 
that $J^{d}_{\rm Serre}(S_X)$ is Serre's Albanese variety of
$(S_X)_{\rm reg} = U \amalg U$, the universal object in the category
of morphisms from $U \amalg U$ to semi-abelian varieties.
We let $J^d(U)$ denote Serre's Albanese variety of $U$ so that
$J^{d}_{\rm Serre}(S_X) = J^d(U) \times J^d(U)$.
Using the above constructions, one checks that there is an exact sequence of 
morphisms of semi-abelian varieties
\begin{equation}\label{eqn:alb-lci-0}
0 \to T \to J^d_{\rm Serre}(S_X) \to J^d(S_X) \to 0,
\end{equation}
where $T$ is an algebraic torus of rank bounded by
the number of components of $D$.

We also have a commutative diagram
\begin{equation}\label{eqn:alb-lci-1}
\xymatrix@C1pc{
\CH^0(D) \ar[d]_{\iota_{-, *}} \ar[r] & NS(X) \ar[d]^{\iota_{-,*}} \\
\Pi_1(S_X) \ar[r] & NS(X_+ \amalg X_-).}
\end{equation}
If we let
$\Pi_{\rm Serre}(X|D) = {\rm Ker}(\CH^0(D) \to NS(X))$,
we obtain a commutative diagram of 1-motives
\begin{equation}\label{eqn:alb-lci-2}
\xymatrix@C.8pc{
0 \ar[d] & 0 \ar[d] \\
[0 \to \Pic^0(X)] \ar[r]^-{\alpha_X} \ar[d]_{\iota_{-,*}} &  
[\Pi_{\rm Serre}(X|D) \to \Pic^0(X)] \ar[d]^{\iota_{-,*}} \\ 
[\Pi(S_X) \to \Pic^0(\ov{S}_X)] \ar[r]^-{\beta_X} \ar[d]_{\iota^*_{+}} &
[\Pi_{\rm Serre}(S_X) \to \Pic^0(\ov{S}_X)] \ar[d]^{\iota^*_+} \\
[\Pi(S_X) \to \Pic^0(X)] \ar[r]^-{\gamma_X} \ar[d] & 
[\Pi_{\rm Serre}(X|D) \to \Pic^0(X)] \ar[d] \\
0 & 0}
\end{equation}
in which the two columns are exact.

We now consider the commutative diagram
\begin{equation}\label{eqn:alb-lci-4}
\xymatrix@C.8pc{
\CH^0(D) \ar[r] \ar[d]_{\wt{\Delta}^*_X} & NS(X) \ar[d]^{\wt{\Delta}^*_X} \\
\Pi_1(S_X) \ar[r] & NS(\ov{S}_X),}
\end{equation}
where we let $\wt{\Delta}^*_X(a) = (a, -a)$. Note that
$\Pi_1(S_X) = \CH^0(D) \oplus  \CH^0(D)$ and 
$NS(\ov{S}_X) = NS(X) \oplus NS(X)$.
It is clear then that $\wt{\Delta}^*_X$ factors through
$\CH^0(D) \to 
{\rm Ker}(\Pi_1(S_X) \xrightarrow{\pi_*} {\rm Div}(S_X))$.
We thus get a morphism of 1-motives
$\wt{\Delta}^*_X: [\Pi_{\rm Serre}(X|D) \to \Pic^0(X)] 
\to  [\Pi(S_X) \to \Pic^0(\ov{S}_X)]$.
Composing this with $\iota^*_+$ in the left column of \eqref{eqn:alb-lci-2},
we get a map $[\Pi_{\rm Serre}(X|D) \to \Pic^0(X)] \to
[\Pi(S_X) \to \Pic^0(X)]$.

\begin{lem}\label{lem:Semi-alb-mod}
The morphism of 1-motives
\[
\iota^*_+ \circ \wt{\Delta}^*_X: [\Pi_{\rm Serre}(X|D) \to \Pic^0(X)] 
\to [\Pi(S_X) \to \Pic^0(X)]
\]
is an isomorphism.
\end{lem}
\begin{proof}
To prove the lemma, we first note that at the level of the abelian varieties,
this map of 1-motives is given by $a \mapsto (a, -a) \mapsto a$, so it is 
actually identity. Thus, we only need to show that the map
$\wt{\Delta}^*_X: \Pi_{\rm Serre}(X|D) \to \Pi(S_X)$ is an isomorphism.

To prove this, let $\Pi_2(S_X) = 
{\rm Ker}(\Pi_1(S_X) \xrightarrow{\pi_*} {\rm Div}(S_X))$.
We then have the commutative diagram
\begin{equation}\label{eqn:Semi-alb-mod-0}
\xymatrix@C.8pc{
\Pi_{\rm Serre}(X|D) \ar[r] \ar[d] & \CH^0(D) \ar[d] 
\ar[r]^-{cl_X} & NS(X) \ar[d] \\
\Pi(S_X) \ar[r] & \Pi_2(S_X) \ar[r]^-{cl_{S_X}} & NS(\ov{S}_X),}
\end{equation}
where the vertical arrows are $\wt{\Delta}^*_X$.
Since $D \amalg D \xrightarrow{\pi} D$ is just the collapse map,
it is immediate from various definitions that any element of
$\Pi_2(S_X)$ must be of the form $\wt{\Delta}^*_X(a)$ for some
$a \in \CH^0(D)$. Furthermore, we have that 
$\iota^*_+ \circ \wt{\Delta}^*_X(a)$ is identity.
It follows that the middle vertical arrow in ~\eqref{eqn:Semi-alb-mod-0}
is an isomorphism. The right vertical arrow is injective and is split by
$\iota^*_+$. 

If we let $\wt{NS(X)}$ and $\wt{NS(\ov{S}_X)}$ be 
the images of ${cl_X}$ and $cl_{S_X}$, respectively, we therefore get a commutative
diagram of short exact sequences
\begin{equation}\label{eqn:Semi-alb-mod-1}
\xymatrix@C.8pc{
0 \ar[r] & \Pi_{\rm Serre}(X|D) \ar[r] \ar[d] & \CH^0(D) \ar[d] 
\ar[r]^-{cl_X} & \wt{NS(X)} \ar[d] \ar[r] & 0 \\
0 \ar[r] & \Pi(S_X) \ar[r] & \Pi_2(S_X) \ar[r]^-{cl_{S_X}} & 
\wt{NS(\ov{S}_X)} \ar[r] & 0,}
\end{equation}
where the middle and the right vertical arrows are isomorphisms.
It follows that the left vertical arrow is an isomorphism too.
This proves the lemma.
 \end{proof}

 \subsection{Albanese homomorphisms for 0-cycles}\label{sec:Alb-hom}
 We continue with the setup of \S~\ref{sec:Alb-double*}.
We let $J^d(X|D):= {\rm Ker}(J^d(S_X) \xrightarrow{\iota^*_-} J^d(X))$, where
the map $\iota^*_-$ is surjective and is split by $\Delta^*_X:
J^d(X) \to J^d(S_X)$.
Using \cite[Prop.~9.7, \S~11.1]{BK}, 
one shows that $J^d(X|D)$ is the universal regular semi-abelian variety
quotient of $\CH_0(X|D)_{\deg 0}$. Furthermore, there is a commutative diagram
of split short exact sequences
\begin{equation}\label{eqn:doubl-ex-seq-*}
\xymatrix@C.8pc{
0 \ar[r] & \CH_0(X|D)_{\deg 0}  \ar[r]^-{p_{+,*}} \ar@{->>}[d]_{\rho_{X|D}} &
\CH^{BK}_0(S_X)_{\deg 0} \ar[r]^-{\iota^*_-} \ar@{->>}[d]^{\wt{\rho}_{S_X}} & 
\CH_0(X)_{\deg 0} \ar@{->>}[d]^{\rho_X} \ar[r] & 0 \\
0 \ar[r] & J^d(X|D)  \ar[r]^-{p_{+,*}} & J^d(S_X) \ar[r]^-{\iota^*_-} &
J^d(X) \ar[r] & 0.}
\end{equation}

Next, recall from \cite[\S 3]{SS} that the semi-abelian Albanese variety $J^d(U)$ of $U$ is the universal regular quotient of $H_0^{sus}(U)_{\deg 0}$.
We write $\rho_U$ for the induced surjection $\rho_U\colon H_0^{sus}(U)_{\deg 0} \surj
J^d(U)$. Composing $\rho_U$ with the canonical map $\lambda_{X|D}$ of 
\eqref{eqn:mod-Suslin}, we get a (surjective) morphism
\[\CH_0(X|D)_{\deg 0}\to J^d(U),\]
and the universality of $J^d(X|D)$ as a regular quotient of $\CH_0(X|D)_{\deg 0}$
induces then a canonical map $\lambda^{alb}_{X|D}: J^d(X|D) \to J^d(U)$.
Using ~\eqref{eqn:alb-lci-2} and \lemref{lem:Semi-alb-mod}, we easily obtain
the following description of various semi-abelian Albanese varieties.

\begin{thm}\label{thm:Semi-abl-Alb}
The canonical map $\CH_0(X|D) \to H^{sus}_0(U)$ induces a commutative
diagram
\begin{equation}\label{eqn:doubl-ex-seq-*-0}
\xymatrix@C.8pc{
\CH_0(X|D)_{\deg 0} \ar@{->>}[r]^-{\lambda_{X|D}} \ar@{->>}[d]_{\rho_{X|D}} &
H^{sus}_0(U)_{\deg 0} \ar@{->>}[d]^{\rho_U} \\
J^d(X|D) \ar[r]^-{\lambda^{alb}_{X|D}} & J^d(U)}
\end{equation}
such that the following hold.
\begin{enumerate}
\item
$\lambda^{alb}_{X|D}$ is an isomorphism of semi-abelian varieties.
\item
$J^d(X|D)$ is the Cartier dual of the 1-motive 
$[\Pi(S_X) \to \Pic^0(X)]$. 
\item
$J^d(U)$ is the Cartier dual of the 1-motive 
$[\Pi_{\rm Serre}(X|D) \to \Pic^0(X)]$.
\item
There is a commutative diagram of short exact sequences of
semi-abelian varieties
\begin{equation}\label{eqn:doubl-ex-seq-*-1}
\xymatrix@C.8pc{
0 \ar[r] & J^d(U) \ar[r]^-{\iota_{+,*}} \ar[d] & J^d(U) \times J^d(U) 
\ar[r]^-{\iota^*_-} \ar[d] & J^d(U) \ar[r] \ar[d] & 0 \\
0 \ar[r] & J^d(X|D) \ar[r]^-{p_{+,*}} & J^d(S_X) \ar[r]^{\iota^*_-} & 
J^d(X) \ar[r] & 0,}
\end{equation}
where the vertical arrows are surjective.
\end{enumerate}
\end{thm}

\subsection{Chow group vs. motivic cohomology of the double}
\label{sec:Relation}
We let $k$ be an algebraically closed field of characteristic $p > 0$.
We let $(X,D)$ be a pair consisting of an integral smooth projective $k$-scheme $X$
of dimension $d\geq 1$ and a reduced effective Cartier divisor $D$ on it. 
Under this setup, we shall identify the Levine-Weibel Chow group of $S_X$ to its
motivic cohomology.
We shall use the following consequence of \cite[Thm.~6.5]{Krishna-2}
which uses our assumption on $D$ in an essential way.
But note that it does not use our assumption on ${\rm char}(k)$.

\begin{lem}$($\cite[Thm.~6.5]{Krishna-2}$)$\label{lem:LW-lci*}
The canonical map $\delta_{S_X} \colon \CH^{LW}_0(S_X) \to \CH_0^{BK}(S_X)$ is an 
isomorphism.
\end{lem}

\begin{lem}\label{lem:LW-MC}
The map $\lambda_{S_X}\colon \CH_0^{BK}(S_X)_\Lambda \to H^{2d}(S_X, \Lambda(d))$ is
surjective and the map $\lambda_{S_X}\colon \CH_0^{BK}(S_X)_\Lambda\{p'\} \to 
H^{2d}(S_X, \Lambda(d))\{p'\}$ is an isomorphism.
\end{lem}
\begin{proof}
In this proof, we shall assume all abelian groups to be tensored with
$\Lambda$.
By construction, there is a commutative diagram (see \cite[(8.13)]{KP})
\begin{equation}\label{eqn:LW-MC-0}
\xymatrix@C1pc{
\sZ_0(S_X,D) \ar@{=}[r] \ar@{->>}[dd] & \sZ_0(S_X,D) \ar@{->>}[dr] \ar[dd] 
& & \\
& & H^{sus}_0(S_X \setminus D) \ar[dl] \\
\CH_0^{BK}(S_X) \ar[r]^-{\lambda_{S_X}} & H^{2d}(S_X, \Lambda(d)),}
\end{equation}
where the map $H^{sus}_0(S_X \setminus D) \to H^{2d}(S_X, \Lambda(d))$ is induced by the localization sequence \eqref{eqn:loc-seq}  associated to the open embedding $S_X\setminus D\to S_X$. 
It follows from \cite[Thm.~5.1]{KP} that this map  is
surjective. This in particular implies that the motivic cohomology group $H^{2d}(S_X, \Lambda(d))$ is generated by the free abelian group on the set of closed points in $S_X\setminus D$. But this is equivalent to $\lambda_{S_X}$ being 
surjective.

To show the isomorphism between the primed torsion subgroups 
(i.e., on the prime-to-$p$-torsion part), we consider the
commutative diagram of split short exact sequences (note that the lines stay 
exact even after applying $(-)\{p'\}$, since \eqref{eqn:doubl-ex-seq} is split)
\begin{equation}\label{eqn:LW-MC-1}
\xymatrix@C1pc{
0 \ar[r] & (\CH_0(X|D))\{p'\} \ar[r] \ar[d] & (\CH_0^{BK}(S_X))\{p'\} 
\ar[r] \ar[d] & (\CH_0(X))\{p'\} \ar[r] \ar[d] & 0 \\
0 \ar[r] & (H^{sus}_0(X \setminus D))\{p'\} \ar[r] & (H^{2d,d}(S_X))\{p'\} \ar[r] &
(H^{2d,d}(X))\{p'\} \ar[r] & 0.}
\end{equation}

The right vertical arrow is an isomorphism as $X$ is smooth. It suffices 
therefore to show that the left vertical arrow is an isomorphism.
For this, we consider the commutative diagram ~\eqref{eqn:doubl-ex-seq-*-0}.
It follows from \thmref{thm:Semi-abl-Alb}(1) that the lower horizontal arrow
in  ~\eqref{eqn:doubl-ex-seq-*-0} is an isomorphism. The left (resp. right) vertical
arrow is an isomorphism on the primed torsion subgroups by
\cite[Thm.~11.5]{BK} (resp. \cite[Thm.~1.1]{SS}). Our assertion now 
follows.
\end{proof}

We can now state and prove the main result of \S~\ref{sec:Alb-double}.

\begin{thm}\label{thm:LW-MC-Main}
 Let $k$ be an algebraically closed field of characteristic $p > 0$.
Let $X$ be a smooth projective $k$-scheme
of pure dimension $d\geq 1$ and $D \subset X$ a reduced effective Cartier
divisor. Then the composition 
\[
\lambda^{LW}_{S_X} \colon \CH^{LW}_0(S_X)_\Lambda \to \CH^{BK}_0(S_X)_\Lambda \to 
H^{2d}(S_X, \Lambda(d))
\] 
is an isomorphism.
\end{thm}
\begin{proof}
  We can assume $X$ is integral.
  The map $\lambda^{LW}_{S_X}$ is surjective by Lemmas~\ref{lem:LW-lci*} and
  ~\ref{lem:LW-MC}. To show it is injective, we first assume
$\Lambda = \Z[\tfrac{1}{p}]$.

It follows from \cite[Thm.~5.9, Lem.~8.1]{KP} that there
is a commutative diagram
\begin{equation}\label{eqn:LW-MC-Main-0}
\xymatrix@C.8pc{
\CH^{LW}_0(S_X)_\Lambda \ar[r] \ar[d]_-{\lambda_{S_X}^{LW}} & 
K_0(S_X)_\Lambda \ar[d] \\
H^{2d,d}(S_X)_\Lambda \ar[r] & KH_0(S_X)_\Lambda,}
\end{equation}
where $KH_*(S_X)$ is Weibel's homotopy $K$-theory.
We remark here that even as \cite[Lem.~8.1]{KP} is proven over the field
of complex numbers, its proof holds over any algebraically closed field.

The right vertical arrow in ~\eqref{eqn:LW-MC-Main-0}
is an isomorphism by \cite{WeibelKH} (see also \cite[Exer.~9.11]{TT}) as $p > 0$.
The kernel of the top horizontal arrow is a torsion group of
exponent $(d-1)!$ by \cite[Cor.~2.7]{Levine-2}. It follows that the 
kernel of $\lambda^{LW}_{S_X}$ is torsion of exponent $(d-1)!$.
We therefore have a short exact sequence
\begin{equation}\label{eqn:LW-MC-Main-1}
0 \to {\rm Ker}(\lambda^{LW}_{S_X})_\Lambda \to 
\CH^{LW}_0(S_X)_\Lambda \to H^{2d,d}(S_X)_\Lambda \to 0,
\end{equation}
where the group on the left is torsion of exponent $(d-1)!$.
On the other hand, since $k$ is algebraically closed and $S_X$ is projective,
one knows that the degree zero part
of $\CH^{LW}_0(S_X)_\Lambda$ is divisible (see the last part in the proof of
\thmref{thm:Mod-fin}). 
Since ${\rm Ker}(\lambda^{LW}_{S_X})$ lies in this subgroup,
it follows from \lemref{lem:LW-MC} that 
${\rm Ker}(\lambda^{LW}_{S_X})_\Lambda$ is also divisible.
The result now follows since a divisible torsion group of bounded exponent
must be zero.

We are now left with proving the injectivity of $\lambda^{LW}_{S_X}$ when $k$ admits
resolution of singularities (even if ${\rm char}(k) > 0$) and $\Lambda = \Z$.
In this case, the above argument shows that $\lambda^{LW}_{S_X}$ is surjective whose
kernel is a $p$-primary torsion group.
We now consider the commutative diagram
\begin{equation}\label{eqn:LW-MC-Main-2} 
\xymatrix@C1pc{
\CH^{LW}_0(S_X)\{p\} \ar[r]^-{\pi^*} \ar[d]_{\rho_{S_X}} &
\CH_0(\ov{S}_X)\{p\} \ar[d]^{\rho_{\ov{S}_X}} \\
J^d(S_X)\{p\} \ar[r]^-{\pi^*} & J^d(\ov{S}_X)\{p\}.}
\end{equation}

It follows from \cite[Thm.~6.4]{Krishna-2} that the vertical arrows are
isomorphisms. On the other hand, we know that $J^d(S_X)$ is a 
semi-abelian variety whose abelian variety quotient is $J^d(\ov{S}_X)$.
Since an algebraic torus over $k$ has no $p$-primary torsion (note
that ${\rm char}(k) = p > 0$), it follows that the lower horizontal
arrow in ~\eqref{eqn:LW-MC-Main-2} is an isomorphism.
In particular, the upper horizontal arrow is also an isomorphism.

In the final step, we consider the commutative diagram
\begin{equation}\label{eqn:LW-MC-Main-3} 
\xymatrix@C1pc{
\CH^{LW}_0(S_X)\{p\} \ar[r]^-{\lambda^{LW}_{S_X}} \ar[d]_{\pi^*} &
H^{2d,d}(S_X)\{p\} \ar[d]^{\pi^*} \\
\CH_0(\ov{S}_X)\{p\} \ar[r]^-{\lambda_{\ov{S}_X}} & 
H^{2d,d}(\ov{S}_X)\{p\}.}
\end{equation}
Since $\ov{S}_X$ is smooth, $\lambda_{\ov{S}_X}$ is an isomorphism.
We have shown above that
the left vertical arrow in ~\eqref{eqn:LW-MC-Main-3} is an isomorphism. It follows that
$\lambda^{LW}_{S_X}$ is injective on the $p$-primary torsion subgroups.
We conclude that ${\rm Ker}(\lambda^{LW}_{S_X})$ is zero.
This completes the proof.
\end{proof}

We end this section with the following question.
Using the diagrams ~\eqref{eqn:alb-lci-1} and ~\eqref{eqn:alb-lci-4},
one shows that there is an 1-motive 
$[\Lambda_{\rm Serre}(X|D) \to \Pic^0(\ov{S}_X)]$ and a map of
1-motives $[\Lambda_{\rm Serre}(X|D) \to \Pic^0(\ov{S}_X)] \to
[\Lambda(S_X) \to  \Pic^0(\ov{S}_X)]$. Let $J^d_{cdh}(S_X)$ denote the
Cartier dual of $[\Lambda_{\rm Serre}(X|D) \to \Pic^0(\ov{S}_X)]$.

\begin{ques}\label{ques:cdh-alb}
Is there a surjective homomorphism $\rho^{cdh}_{S_X}:H^{2d,d}(S_X)_{\deg 0} \surj 
J^d_{cdh}(S_X)$ and a commutative diagram
\begin{equation}\label{eqn:cdh-alb-1}
\xymatrix@C1pc{
\CH^{LW}_0(S_X)_{\deg 0} \ar[r]^-{\lambda_{S_X}} \ar[d]_{\rho_{S_X}} & 
H^{2d,d}(S_X)_{\deg 0} \ar[d]^{\rho^{cdh}_{S_X}} \\
J^d(S_X) \ar[r] & J^d_{cdh}(S_X)}
\end{equation}
such that the vertical arrows are the universal regular quotient maps?
\end{ques}

\section{Proofs of Theorems~\ref{thm:Main-2} and
~\ref{thm:Main-2-fin}}\label{sec:MS**}
In this section, we shall prove Theorems~\ref{thm:Main-2} and
~\ref{thm:Main-2-fin}. We shall prove the two parts of Theorem~\ref{thm:Main-2}
in separate subsections for the sake of clarity of exposition.
We begin with the proof of its part (1).

\subsection{Positive characteristic case}\label{sec:Th-2}
Let $k$ be an algebraically closed field of characteristic $p > 0$.
Let $(X,D)$ be a pair consisting of an integral smooth projective $k$-scheme $X$
of dimension $d\geq 1$ and an effective (but not necessarily reduced)
Cartier divisor $D$ on it. We let $U = X \setminus D$. 
We need a lemma and a proposition.

\begin{lem}\label{lem:Kernel-mod}
  Let $E \subset X$ be an effective Cartier divisor which is contained in $D$ and
  whose support coincides with that of $D$. Then the kernel of the canonical surjection
  $\CH_0(X|D) \surj \CH_0(X|E)$ is a $p$-primary torsion group.
\end{lem}
\begin{proof}
We shall prove the lemma by induction on $d$.
If $d \le 2$, the cycle class map $\CH_0(X|D) \to H^{d}_\zar(X, \sK^M_{d, (X,D)})$
is an isomorphism by \cite[Thm.~1.8]{BK}, where $\sK^M_{n, (X,D)} =
\Ker(\sK^M_{n, X} \surj \sK^M_{n,D})$ is the Zariski sheaf of relative Milnor $K$-groups.
It suffices therefore to show that the kernel of the canonical surjection
$H^{d}_\zar(X, \sK^M_{d, (X,D)}) \surj H^{d}_\zar(X, \sK^M_{d, (X,E)})$ is
a $p$-primary torsion group.
Using the long exact cohomology sequence associated to the exact sequence of
Zariski sheaves
\[
  0 \to \sK^M_{d, (X,D)} \to \sK^M_{d, (X,E)} \to \sK^M_{d, (D,E)} \to 0,
\]
our assertion is implied by the claim that $H^{d-1}_\zar(D, \sK^M_{d, (D, E)})$ is a
$p$-primary torsion group (in fact, of bounded exponent).
But this is well known (e.g., see \cite[Lem.~2.6]{GKris}).

We now assume $d \ge 3$.
Let $\nu \colon C \to X$ be a finite morphism from a regular integral
projective curve whose image is not contained in $D$
and let $f \in {\rm Ker}(\sO^{\times}_{C,E'} \surj \sO^{\times}_{E'})$,
where $E' = \nu^{-1}(E)$. We need to show that $\nu_*(\divf(f)) \in
\CH_0(X|D)$ is killed by a power of $p$.

We can get a factorization $C \stackrel{\nu'}{\inj} 
\P^n_X \xrightarrow{\pi} X$ of $\nu$, where $\nu'$ is a closed
immersion and $\pi$ is the canonical projection. It is
clear that $\nu'_*(\divf(f))$ dies in $\CH_0(\P^n_X|\P^n_E)$ and
$\nu_*(\divf(f)) = \pi_*(\nu'_*(\divf(f)))$.
Using the push-forward map on the Chow groups with modulus, 
it suffices therefore to show that 
$\nu'_*(\divf(f))$ is killed by some power of $p$ in $\CH_0(\P^n_X|\P^n_D)$.
We can thus assume that $\nu \colon C \to X$ is a closed immersion.

We now fix a closed embedding $X \inj \P^N_k$. We can use \cite[Thm.~7]{AK} 
to find a hypersurface $H \subset \P^N_k$ containing $C$ and not containing $X$ 
such that the scheme theoretic intersection $X \cap H$ 
has the property that it is smooth and contains
no irreducible component of $D$. By \cite[Cor.~3.13]{GK}, we can assume
that $X \cap H$ is integral. We let $Y = X \cap H$, $D_Y = D \cap H$
and $E_Y = E \cap H$.

Let $C \stackrel{\nu'}{\inj} Y \stackrel{\iota}{\inj} X$ 
be the factorization of $\nu$.
It follows from the choice of $Y$ that 
$\nu'_*(\divf(f))$ dies in $\CH_0(Y| E_Y)$. This implies by induction that
$\nu'_*(\divf(f)) \in \CH_0(Y|D_Y)$ is killed by a power of $p$.
It follows that $\nu_*(\divf(f)) = \iota _* (\nu'_*(\divf(f)))$
is killed by a power of $p$ in $\CH_0(X|D)$. This concludes the proof.
\end{proof}

\enlargethispage{25pt}

\begin{prop}\label{prop:Kernel-torsion}
The kernel of the map $\lambda_{X|D} \colon \CH_0(X|D) \surj H^{sus}_0(U)$ is a 
$p$-primary torsion group. Furthermore, the map
$\lambda_{X|D} \colon \CH_0(X|D)\{\ell\} \surj H^{sus}_0(U)\{\ell\}$
is an isomorphism for any prime $\ell \neq p$.
\end{prop}
\begin{proof}
By \lemref{lem:Kernel-mod}, we can assume that $D$ is reduced.
  Now, the second part of the proposition is already shown
  in the proof of \lemref{lem:LW-MC} (see the argument after ~\eqref{eqn:LW-MC-1}).
  To prove the first part, we let
  $S_X$ be the double of $X$ along $D$ and consider the commutative diagram
  (see ~\eqref{eqn:LW-MC-1} again)
  \begin{equation}\label{eqn:Kernel-torsion-0}
\xymatrix@C.8pc{
  \CH_0(X|D)[\tfrac{1}{p}] \ar[r] \ar[d]_-{\lambda_{X|D}} &
  \CH_0^{BK}(S_X)[\tfrac{1}{p}] \ar[d]^-{\gamma_{S_X}} \\
  H^{sus}_0(U)[\tfrac{1}{p}] \ar[r] &  H^{2d,d}(S_X)[\tfrac{1}{p}].}
\end{equation}
The horizontal arrows are split injective, and the right vertical arrow is an
isomorphism by \thmref{thm:LW-MC-Main}. It follows that the left vertical
arrow is also an isomorphism. This implies our assertion.
\end{proof}

We can now prove \thmref{thm:Main-2}(1). We restate it for convenience.

\begin{thm}\label{thm:torsion}
Let $X$ be a smooth projective scheme over
an algebraically closed field $k$ of characteristic $p > 0$.
Let $D \subset X$ be a reduced effective Cartier divisor with complement $U$.
Then the canonical map 
\[
\lambda_{X|D} \colon \CH_0(X|D) \to H^{sus}_0(U)
\]
is an isomorphism.
\end{thm}
\begin{proof}
  We can assume $X$ to be integral.
By \propref{prop:Kernel-torsion}, we only need to show that
$\lambda_{X|D}$ induces an isomorphism between the torsion subgroups.
For this, we use the commutative diagram~\eqref{eqn:doubl-ex-seq-*-0}. 
It follows from \thmref{thm:Semi-abl-Alb}(1) that the bottom horizontal
arrow in this diagram is an isomorphism. The left (resp. right)  vertical arrow is
an isomorphism on the torsion subgroups by \cite[Thm.~1.2]{Krishna-2}
(resp. \cite[Thm.~1.7]{GK-1}).
This implies that the same is true for the top horizontal arrow.
\end{proof}

\subsection{Characteristic zero case}\label{sec:Char-0}
Let $k$ be any field (not necessarily algebraically closed) of characteristic zero.
We let $\Omega_k$ be the module of absolute K{\"a}hler differentials of $k$.
The second part of \thmref{thm:Main-2} is an easy consequence of the following
general result.

\begin{thm}\label{thm:Failure}
  Let $\ol{C} \subset \P^2_k$ be
the cuspidal curve given by the homogeneous equation
$x_0x^2_1-x^3_2 = 0$. Then $\CH_0(\P^2_k|\ol{C}) \cong \Omega_k \oplus \Z$ and
$H^{sus}_0(\P^2_k \setminus \ol{C}) \cong \Z$.
\end{thm}
\begin{proof}
  We shall freely use the duality results from \cite{VoevTriangCat}.
  We drop the subscript $k$ from the notations of $\P^n_k$ and $\A^n_k$.

  Let us first prove that $H^{sus}_0(\P^2_k\setminus \ol{C}) \cong \Z$. Let
  $C_\infty = (0:1:0)$ be the point at infinity and $P = (1:0:0)$ the singular
  point of $\ol{C}$. We write 
$C = \ol{C} \setminus C_\infty = (\P^2_k \setminus H_\infty) \cap \ov{C}$, where 
$H_\infty \cong \P^1_k$ is the line $\{X_0 =0\}$.
We claim that there is a split exact sequence
    \begin{equation}\label{eqn:RussellChow}
      0\to k^\times \to H_{c }^4(\A^2 \setminus C, \Z(2)) \to \Z\to 0
    \end{equation}

Indeed,  letting $U = \A^2\setminus C$, the localization triangle
        \[M^c(C) \to M^c(\A^2)\to M^c(U)\to M^c(C)[1]\]
    in ${\mathbf{DM}}^{\rm eff}(k, \Z)$ gives rise to 
an exact sequence of cohomology with compact support
        \begin{equation}\label{eq:computehom} 
          \cdots \to H_{c}^3(C, \Z(2)) \to H_{c}^4(U, \Z(2)) \to H_{c}^4(\A^2, \Z(2)) \to
          H_{c}^4(C, \Z(2)).
        \end{equation}
By duality and homotopy invariance, we have 
\[
H_{c}^4(\A^2, \Z(2)) =
H_0(\A^2, \Z(0)) = H_0(k, \Z(0)) =\Z.
\]
Since $k$ is infinite and $U$ is a rational variety, we can choose
$x \in U(k)$. We then get a sequence of maps
\[
  H^0(k(x), \Z(0)) \to H_{c}^4(U, \Z(2)) \to H_{c}^4(\A^2, \Z(2)) \to
  H^0(k, \Z(0)),
\]
where the first arrow is the Gysin homomorphism and the last arrow is the
push-forward map between cohomology with compact support induced by the
flat morphism $\A^2 \to \Spec(k)$. Since the composition is known to be
an isomorphism, we get a split surjection $H_{c }^4(U, \Z(2)) \surj \Z$.

We are left to compute $H_{c}^{3}(C, \Z(2))$. Comparing the exact triangles for $C$ and for $\A^1$, again we get $H_{c}^{3}(C, \Z(2)) = H_{c}^{3}(\A^1, \Z(2))$. By duality and homotopy invariance, the latter is isomorphic to $H_{-1}(k, \Z(-1)) = H^1(k, \Z(1)) = k^\times$. The long exact sequence \eqref{eq:computehom} then gives 
\[ H_{c}^{3}(\A^2, \Z(2)) \to k^\times \to  H_{c}^4(U, \Z(2)) \to \Z \to 0.\]
Since $H_{c}^{3}(\A^2, \Z(2)) = H_1(\A^2, \Z) = H_1(k, \Z) = 0$ and 
$H_0(U, \Z(0)) \cong H_{c}^4(U, \Z(2))$ by duality, we finally get \eqref{eqn:RussellChow}.

Next, note that $H_\infty\setminus C_\infty = \A^1$. Looking at the triangle in 
${\mathbf{DM}}^{\rm eff}(k, \Z)$:
\[ M^c(H_\infty\setminus C_\infty) \to M^c(\P^2\setminus \ol{C}) \to M^c(U) \to M^c(H_\infty\setminus C_\infty)[1],\]
we get 
\[k^\times = H_{c}^{3}(\A^1, \Z(2)) \to H_{c}^4(U, \Z(2)) \to  H_{c}^4(\P^2\setminus \ol{C}, \Z(2)) \to 0.\]
Since $H_{c}^4(U, \Z(2)) = k^\times\oplus \Z$ by \eqref{eqn:RussellChow}, it's clear by construction that $H_{c}^4(\P^2\setminus \ol{C}, \Z(2)) = H_0^{sus}(\P^2 \setminus \ol{C}, \Z(0))=\Z$.

We now turn to the first statement of the theorem. By \cite[Thm.~1.1]{BKS},
we have an identification $\CH_0(\P^2|\ol{C}) \cong
H^2_{\rm Zar}(\P^2, \mathcal{K}^M_{2, (\P^2, \ol{C})})$.
Since the relative $K_0$-group $K_0(\P^2, \ol{C})$ decomposes as
$H^2_{\rm Zar}(\P^2, \mathcal{K}^M_{2, (\P^2, \ol{C})}) \oplus \Pic(\P^2, \ol{C})$ and the
group  $\Pic(\P^2, \ol{C})$ vanishes by direct inspection, it is enough to compute
$K_0(\P^2, \ol{C})$.

This can be done explicitly by means of the long localization sequence for $K$-theory
with support. In particular, there is a short exact sequence
\[0\to \Z \to K_0(\P^2, \ol{C}) \to K_0(\A^2, C) \cong SK_1(C)\to 0\]
and we are left to show that $SK_1(C) \cong \Omega_{k}$. This is classical, and
can be done by descent along the normalization morphism $\nu\colon \A^1\to C$ using
the $K_1(A,B,I)$ computation of \cite{GW}. We leave the details to the interested
reader.
\end{proof}

\subsection{Kerz-Saito vs. Russell's Chow groups with modulus}
\label{sec:Russell}
    There exists another notion  different from the one used in this paper (and specifically used in the previous theorem) of Chow group of 0-cycles with 
modulus. This Chow group was introduced by H. Russell \cite{RussellANT} in the study of a higher dimensional analogue of the generalized Jacobian of Rosenlicht-Serre. Russell's Chow group with modulus is defined as the cokernel of the map
     \[\divf\colon \bigoplus_{\phi_C\colon C\to X}G(C, \phi_C^*(D-D_{\rm red}) + (\phi_C^*(D))_{\rm red})\to \sZ_0(X,D).\]

     In particular, in view of \thmref{thm:suslinhomology}, it is clear that it agrees with $H_0^{sus}(X\setminus D)$ when $D$ is reduced and $X$ is proper. Theorem \ref{thm:Failure} then exhibits explicitly an example 
(at least in characteristic zero)
for which Russell's Chow group and the Kerz-Saito Chow group do not agree. 
Several authors have enquired in the past about this comparison.
Note that in view of  \cite[Thm. 1.1]{BKS}, it follows that Russell's Chow group does not satisfy Bloch's formula, even for $D$ reduced. In particular,
it may not admit a cycle class map to relative $K$-theory.

\subsection{Finite coefficient case}\label{sec:Main2-III}
We now prove \thmref{thm:Main-2-fin}. We restate it for convenience.

\begin{thm}\label{thm:M1}
  Let $k$ be a perfect field and $X$ a smooth projective $k$-scheme.
  Let $D \subset X$ be an effective Cartier divisor with complement $U$.
  Assume that the irreducible components of $D_\red$ are regular.
  Then the map
 \[
    \lambda_{X|D} \colon \CH_0(X|D)_R \to H^{sus}_0(U)_R
  \]
  is an isomorphism, where
  $R = [\tfrac{1}{p}]$ if ${\rm char}(k) = p > 0$
  and ${\Z}/m$ with $m \neq 0$ if ${\rm char}(k) = 0$.
\end{thm}
\begin{proof}
When $\dim(X) = 1$, we can use \cite[Thm.~1.3(2)]{Miyazaki} to
assume that $D$ is reduced. In the latter case, the theorem holds integrally
as one can directly check. We therefore assume $d \ge 2$.
We can also assume $X$ is integral.
Since both sides of ~\eqref{eqn:mod-Suslin} are quotients of the free abelian
group $\sZ_0(U)$, we need to show that
any relation imposed on $\sZ_0(X|D)$ to define $H^{sus}_0(U)$ lies in the image of
$\sR_0(X|D) \to {\sZ_0(U)}_R$, where $\sR_0(X|D) = {\rm Ker}(\sZ_0(U) \surj
\CH_0(X|D))$. 

So let $C \subset X$ be an integral curve not contained in $D$ and let
$\nu \colon C^N \to C \inj X$ be the composite map, where $C^N$ is the normalization
of $C$. Let $f \in k(C^N)^{\times}$ be a rational 
function on $C^N$ which is regular in an affine neighborhood of $\nu^{-1}(D)$
such that $f(x) = 1$ for every $x \in \nu^{-1}(D)$. We need to show that
$\divf(f)$ dies in ${\CH_0(X|D)}_R$.

Since $\nu$ is finite, we can find a factorization 
\begin{equation}\label{eqn:M1-0}
\xymatrix@C1pc{
C^N \ar[r]^-{\phi} \ar[dr]_{\nu} & \P^n_X \ar[d]^{\pi}\\
& X,}
\end{equation}
where $\phi$ is a closed immersion and $\pi$ is the canonical projection.
Since $\pi$ is smooth, the irreducible components of
$D' := \P^n_D$ are regular. Furthermore, $C^N \not\subset D'$.
Since $\pi_*(\divf(f)_{C^N}) = \divf(f)_C = \divf(f)$ 
under the proper push-forward map
$\pi_* \colon \sZ_0(\P^n_U) \to \sZ_0(U)$ and since
$\pi_*(\sR_0(\P^n_X|D')) \subset \sR_0(X|D)$ (see e.g., \cite[\S~2]{KPv} or 
\cite[Lem.~2.7]{BS}), 
it suffices to show that
$\divf(f)_{C^N}$ dies in ${\CH_0(\P^n_X|\P^n_D)}_R$.
Up to replacing $X$ with $\P^n_X$ and $C$ with $C^N$, we can therefore assume that
$C \subset X$ is regular so that $\nu$ is a closed immersion of regular
schemes.

Since $D_\red$ is reduced with regular irreducible components, we can apply 
\cite[Prop.~A.6]{SaitoSatoAnnals} to find a finite sequence of blow-ups
$\phi \colon \wt{X} \to X$ over the closed points of $D \cap C$ such that
the scheme theoretic inverse image $\wt{D} := \wt{X} \times_X D$ satisfies the
following.
\begin{enumerate}
\item
$\wt{D}_\red$ is reduced and its irreducible components are regular. 
\item
The strict transform $\wt{C}$ of $C$ is regular and intersects $\wt{D}_\red$
only in the regular locus of $\wt{D}_\red$ and transversely.
\end{enumerate}

Since $\phi$ is proper, we have a commutative diagram
\begin{equation}\label{eqn:M1-1} 
\xymatrix@C1pc{
\sZ_0(U) \ar@{->>}[r] \ar[d]_{\phi_*} \ar[d]^-{\cong} & 
\CH_0(\wt{X}|\wt{D}) \ar[d]^{\phi_*} \\
\sZ_0(U) \ar@{->>}[r] & \CH_0({X}|{D}),}
\end{equation}
where $\phi_*$ is the push-forward map of the 0-cycles groups.
Since $\phi \colon \wt{C} \to C$ is an isomorphism, we see that
$f \in k(\wt{C})^{\times}$
such that $\phi_*(\divf(f)_{\wt{C}} ) = \divf(f)_{C} = \divf(f)$. 
Moreover, $f$ is regular in a 
neighborhood of $\wt{D} \cap \wt{C}$ with $f(x) = 1$ for every $x \in
\wt{D} \cap \wt{C} \subset \phi^{-1}(D \cap C)$.
In particular, 
$\divf(f)_{\wt{C}} \in {\rm Ker}(\sZ_0(U) \surj H^{sus}_0(U))$.

It follows from ~\eqref{eqn:M1-1} 
that $\divf(f)_C$ will die in ${\CH_0(X|D)}_R$ if we can show that 
$\divf(f)_{\wt{C}}$  dies in ${\CH_0(\wt{X}|\wt{D})}_R$.
We can therefore assume that our original curve $C \subset X$ has the
property that it is regular and it intersects $D_\red$ only in the 
regular locus of $D_\red$ and transversely.
But it is immediate from various definitions that 
in this case, $f \in k(C)^{\times}$ is regular in a neighborhood of $E$
and $f(x) = 1$ for all $x \in E$ if and only if 
$f \in {\rm Ker}(\sO^{\times}_{C, E} \to \sO^{\times}_{E})$, where 
$E = C \times_X D_\red$.
The latter condition implies that $\divf(f) \in \sR_0(X|D_\red)$.
That is, $\divf(f)$ dies in $\CH_0(X|D_\red)$. 
By \cite[Thm.~1.3(2)]{Miyazaki}, this implies that
$\divf(f)$  dies in ${\CH_0({X}|{D})}_R$.
This concludes the proof.
\end{proof}

\section{Bloch's torsion theorem for Suslin homology}
\label{sec:Suslin-tor}
The goal of this section is to prove \thmref{thm:Main-7}.
We shall first establish some preliminary results of independent
interest.

\subsection{A torsion theorem for the double}
\label{sec:Sing-tor}
We fix an algebraically closed field $k$ of characteristic zero and a prime $\ell$.
Let $X$ be an integral smooth projective surface over $k$ and let
$D \subset X$ be a reduced effective Cartier divisor. We let $S_X$ be the
double of $X$ along $D$.

\begin{lem}\label{lem:Double-0}
  The canonical map
  \[
    \lambda_{S_X} \colon \CH^{LW}_0(S_X)\{\ell\} \to H^4(S_X, \Z(2))\{\ell\}
  \]
  is an isomorphism.
\end{lem}
\begin{proof}
  By a Lefschetz principle argument, we can assume that $k = \C$.
  In this case, the Albanese homomorphism $\rho_{S_X} \colon
  \CH^{LW}_0(S_X)_{\deg 0} \to J^2(S_X)$ has a factorization
  \[
    \CH^{LW}_0(S_X)_{\deg 0} \xrightarrow{\lambda_{S_X}}  H^4(S_X, \Z(2))_{\deg 0}
    \xrightarrow{\rho'_{S_X}} J^2(S_X)
  \]
  by \cite[Thm.~8.4]{KP} such that $\rho'_{S_X}$ is an isomorphism on the
  torsion subgroups. On the other hand,
  the main result of \cite{BPW} asserts that $\rho_{S_X}$ is an
  isomorphism on the torsion subgroups. This forces $\lambda_{S_X}$ to be an
  isomorphism on the torsion subgroups.
\end{proof}

By  \cite[Thm.~3.6]{Krishna-2}, there is a canonical isomorphism
$H^3_\et(S_X, {\Q_\ell}/{\Z_\ell}(2)) \xrightarrow{\cong} \CH^{LW}_0(S_X)\{\ell\}$.
One has a similar isomorphism for $X$ by Bloch.
Let $\gamma_{S_X}$ be the composition of this isomorphism with $\lambda_{S_X}$.
We denote the corresponding map for $X$ by $\gamma_X$.

\begin{lem}\label{lem}\label{lem:Double-1}
  There is a commutative diagram
  \begin{equation}\label{eqn:Double-1-0}
    \xymatrix@C.8pc{
      H^3_\et(S_X, {\Q_\ell}/{\Z_\ell}(2)) \ar[r]^-{\iota^*_-} \ar[d]_-{\gamma_{S_X}} &
      H^3_\et(X, {\Q_\ell}/{\Z_\ell}(2)) \ar[d]^-{\gamma_{X}} \\
      H^4(S_X, \Z(2))\{\ell\} \ar[r]^-{\iota^*_-} &
       H^4(X, \Z(2))\{\ell\}}
  \end{equation}
  such that the vertical arrows are isomorphisms.
\end{lem}
\begin{proof}
By \cite[Thm.~3.6]{Krishna-2}, there is a canonical isomorphism
  $\gamma_{S_X} \colon H^3_\et(S_X, {\Q_\ell}/{\Z_\ell}(2)) \xrightarrow{\cong}
  \CH^{LW}_0(S_X)\{\ell\}$ (this is true for any complete surface).
  Moreover, it was shown in \cite[Thm.~5.4]{Krishna-2}
  (see diagram (5.13) of op. cit.) that $\gamma_{S_X}$ and $\gamma_X$ are compatible
  with $\iota^*_-$. Since $\lambda_{S_X}$ and $\lambda_X$ are clearly compatible with
  $\iota^*_-$, our assertion follows by \lemref{lem:Double-0}.
\end{proof}

\begin{lem}\label{lem:Double-2}
Let $U=X\setminus D$.  There is a canonical isomorphism
  \[
    \gamma_U \colon H^3_{c, \etl}(U, {\Q_\ell}/{\Z_\ell}(2))
    \xrightarrow{\cong} H^{sus}_0(U)\{\ell\}.
  \]
\end{lem}
\begin{proof}
  We have a commutative diagram of (split) exact sequences
  \begin{equation}\label{eqn:Double-2-0}
    \xymatrix@C.8pc{
    0 \ar[r] & H^3_{c, \etl}(U, {\Q_\ell}/{\Z_\ell}(2)) \ar[r]^-{p_{+, *}}
    \ar@{.>}[d] & H^3_\et(S_X, {\Q_\ell}/{\Z_\ell}(2)) \ar[r]^-{\iota^*_-}
    \ar[d]^-{\gamma_{S_X}} & H^3_\et(X, {\Q_\ell}/{\Z_\ell}(2)) \ar[d]^-{\gamma_{X}}
    \ar[r] & 0 \\
    0 \ar[r] & H^{sus}_0(U)\{\ell\} \ar[r]^-{p_{+, *}}  & 
    H^4(S_X, \Z(2))\{\ell\} \ar[r]^-{\iota^*_-} & H^4(X, \Z(2))\{\ell\}
    \ar[r] & 0.}
  \end{equation}  
Using this diagram, we conclude the proof by \lemref{lem:Double-1}.
\end{proof}

\vskip .3cm

\subsection{Proof of \thmref{thm:Main-7}}\label{sec:M7-pf}
Let $k$ be an algebraically closed field and
$U$ a quasi-projective $k$-scheme of pure dimension $d \ge 0$.
Let $U \inj X$ be a dense open immersion such that $X$ is a smooth projective
$k$-scheme. Let $\ell$ be a prime different from ${\rm char}(k)$.
We can now prove \thmref{thm:Main-7}. We restate it for convenience.

\begin{thm}\label{thm:torsion-0}
 There is an isomorphism
\begin{equation}\label{eqn:torsion-0-1}
  \vartheta_U \colon
  H^{sus}_0(U)\{\ell\} \xrightarrow{\cong} H^{2d-1}_{c, \etl}(U, {\Q_\ell}/{\Z_\ell}(d)).
\end{equation}
\end{thm}
\begin{proof}
We can assume $d \ge 1$ because the theorem is trivial otherwise.
We can assume $U$ to be integral. 
Let $D \subset X$ be the complement of $U$ with reduced closed subscheme structure.
Let $p \ge 1$ be the exponential characteristic of $p$.
We consider several cases.

Assume first that $p > 1$ and $\dim(D) \le d-2$.
Let $\Lambda = \Z[\tfrac{1}{p}]$.
Since $A\{\ell\} \xrightarrow{\cong} A_\Lambda\{\ell\}$ for any abelian
group $A$,
we can tensor both sides of ~\eqref{eqn:torsion-0-1} with $\Lambda$.
We then have the localization exact sequence (see ~\eqref{eqn:loc-seq} and
the discussion below it)
\[
H^{2d-1}(D, \Lambda(d)) \to H^{sus}_0(U)_\Lambda \to \CH_0(X)_\Lambda \to
H^{2d}(D, \Lambda(d)).
\]
The first and the last terms of this exact sequence vanish by \cite[Thm.~5.1]{KP},
and hence we get $H^{sus}_0(U)_\Lambda \xrightarrow{\cong} \CH_0(X)_\Lambda$.
Using the localization sequence for {\'e}tale cohomology and repeating
the same argument, we also get 
$H^{2d-1}_{c, \etl}(U, {\Q_\ell}/{\Z_\ell}(d)) \xrightarrow{\cong}
H^{2d-1}_{\etl}(X, {\Q_\ell}/{\Z_\ell}(d))$. We therefore reduce to the case
when $U$ is smooth and projective, where Bloch's theorem applies.

We next assume that $p > 1$ and $\dim(D) = d-1$.
Let $D' \subset D$ be the union of irreducible components of $D$ which have
dimension $d-1$ and let $U' = X \setminus D'$.
We have an exact sequence
\[
  H^{2d-1}_c(U' \setminus U, \Lambda(d)) \to H^{sus}_0(U)_\Lambda \to
  H^{sus}_0(U')_\Lambda \to H^{2d}_c(U' \setminus U, \Lambda(d)). 
\]
The two end terms vanish by \cite[Lem.~6.2]{KP-1} because $\dim(U' \setminus U)
\le d-2$. Using the localization sequence for {\'e}tale cohomology and repeating
the same argument, we get
$H^{2d-1}_{c, \etl}(U, {\Q_\ell}/{\Z_\ell}(d)) \xrightarrow{\cong}
H^{2d-1}_{c, \etl}(U', {\Q_\ell}/{\Z_\ell}(d))$. It suffices therefore to prove the
theorem for $U'$. We can therefore assume that $D$ has pure dimension $d-1$.
Since $X$ is smooth, $D$ is necessarily a Cartier divisor on $X$.
In this case, we can combine \thmref{thm:torsion} and 
\cite[Thm.~1.1]{Krishna-2} to conclude our result.

We now assume $p =1$.
The proof in the positive characteristic case remains valid
in characteristic zero if $d \le 1$. However, it breaks down when $d \ge 2$ because
\thmref{thm:torsion} is no more available with us.
We shall adopt a different
strategy which involves revisiting the proof of \cite[Thm.~1.1]{Krishna-2}.

We have already argued above that $D$ can be
assumed to be a Cartier divisor on $X$. This part of the argument does not
involve ${\rm char}(k)$. We shall assume that $X$ is integral and prove
the theorem by induction on $d$.
When $d = 2$, we can apply \lemref{lem:Double-2}.

We now assume $d \ge 3$. We fix a closed embedding $X \inj \P^n_k$.
By \cite[Thm.~1]{AK}, we can find a hypersurface $H \subset \P^n_k$ of
sufficiently high degree such that the scheme theoretic intersection $Y = X \cap H$ 
satisfies the following.
\begin{enumerate}
\item
  $Y$ is smooth and irreducible of dimension $d-1$.
\item
  $Y$ contains no component of $D$.
\item
  $Y \cap U \neq \emptyset$.
\end{enumerate}

It follows from (1) that $Y$ is integral. We let $V = Y \cap U$ and let
$\tau \colon V \inj U$ be the inclusion. We consider the diagram
\begin{equation}\label{eqn:torsion-0-2}
\xymatrix@C.8pc{
H^{2d-3}_{c, \etl}(V, {\Q_\ell}/{\Z_\ell}(d-1))  \ar[r]^-{\gamma_V} \ar[d]_-{\tau_*} &
H^{sus}_0(V)\{\ell\}  \ar[d]^-{\tau_*} \\
H^{2d-1}_{c, \etl}(U, {\Q_\ell}/{\Z_\ell}(d)) \ar@{.>}[r] &
H^{sus}_0(U)\{\ell\}}
\end{equation}
The left vertical arrow is an isomorphism by \cite[Prop.~5.1]{GK-1}.
Since $\gamma_V$ exists and is an isomorphism by induction, it remains to show that
the right vertical arrow is an isomorphism to finish the proof.

We now look at the commutative diagram
\begin{equation}\label{eqn:torsion-0-3}
\xymatrix@C.8pc{
H^{sus}_0(V)\{\ell\} \ar[r]^{\rho_V} \ar[d]_-{\tau_*} &
J^{d-1}(V)\{\ell\}  \ar[d]^-{\tau_*} \\
H^{sus}_0(U)\{\ell\} \ar[r]^{\rho_U} & J^{d}(U)\{\ell\},}
\end{equation}
where the horizontal arrows are the Albanese maps of
\thmref{thm:Semi-abl-Alb}. The horizontal arrows are isomorphisms
by \cite[Thm.~1.1]{SS} and the right vertical arrow is an isomorphism by
\cite[Thm.~7.3]{GK-1}. This forces the left vertical arrow to be an
isomorphism.
\end{proof}

\begin{remk}\label{remk:SS}
  One of the referees remarked that a proof of \thmref{thm:torsion-0} can also be
  obtained using the Roitman torsion theorem of Spie{\ss} and Szamuely \cite{SS} together with Poincar\'e duality.
From a different perspective, one could use   \thmref{thm:torsion-0} to show that  the torsion theorem
  of \cite{SS} in positive characteristic is a direct consequence of
  Bloch's theorem together with the analogous result for singular varieties.
  \end{remk}

\section{The finiteness theorems}
\label{sec:Finiteness}
In this section, we shall derive \thmref{thm:Main-4} from 
some general finiteness results for motivic cohomology groups.
Let $k$ be a field. Let $m \ge 1$ be an integer which is invertible in $k$,
and $p \nmid m$ if $k$ is a $p$-adic field. Set $\Lambda' = {\Z}/m$. 
Let $X$ be a smooth projective $k$-scheme.

If $k$ is algebraically closed, then it is classically
known that ${\CH_0(X)}/m \cong \Lambda'$. 
Otherwise, the structure of ${\CH_0(X)}/m$ is not fully understood.
If $k$ is finite, then the main result of unramified class field
theory in geometric case due to Kato-Saito \cite{Kato-Saito-1} shows that
${\CH_0(X)}/m$ is finite. If $k$ is a $p$-adic field, a conjecture of
Colliot-Th{\'e}l{\`e}ne predicted that ${\CH_0(X)}/m$ is finite.
This conjecture was settled by Saito and Sato in \cite{SaitoSatoAnnals}
(see also \cite[Prop. 9.3]{EKW}).

Let us now recall that there is an isomorphism 
${\CH_0(X)} \cong H^{sus}_0(X)$. One therefore concludes that
${H^{sus}_0(X)}/m$ is finite if $k$ is either of the above three types.
Suppose now that $U$ is a smooth quasi-projective $k$-scheme 
which is not necessarily complete. One can then ask if
${H^{sus}_0(X)}/m$ is finite. One can similarly ask
if ${\CH_0(X|D)}/m$ is finite if $D$ is an effective Cartier divisor on
a smooth projective scheme $X$ over $k$. To our knowledge, no answer is known about
these questions yet. We shall answer these in this section.

\subsection{A general finiteness result}\label{sec:Fin-gen}
In this subsection, we prove a general result about the finiteness of motivic
cohomology groups of possibly singular schemes.

\begin{lem}\label{lem:smooth-proj-fin}
Let $k$ be either a finite, or a $p$-adic or an algebraically closed field.
Let $X$ be a smooth projective $k$-scheme of dimension $d \ge 0$.
Then $H^{2d+i}(X, \Lambda'(d+i))$ is finite for every $i \ge 0$.
\end{lem}
\begin{proof}
If $X$ is integral of dimension $d$, then $H^{2a-b}(X, \Z(b)) \cong
\CH^a(X, b) = 0$ for $a > d+b$. This allows us to assume that
$X$ has pure dimension $d$. Subsequently, we can assume that $X$ is integral.
Using the above discussion of $i = 0$ case, we can assume that $i \ge 1$.

We now note that over any field $k$, there is a canonical isomorphism
$H^{2d+i}(X, \Lambda'(d+i)) \cong {\CH^{d+i}(X, i)}/m$.
Furthermore, Bloch's spectral sequence $H^p_\zar(X, {\sC}{\sH}^q(X,i))
\Rightarrow \CH^{p+q}(X,i)$ and 
the isomorphisms $\CH^p(\sO_{X,x}, p) \cong
H^p(\sO_{X,x}, \Z(p)) \cong K^M_p(\sO_{X,x})$
(the latter is Gabber-Kerz improved Milnor $K$-theory, \cite{Kerz10})
together yield an exact sequence 
\begin{equation}\label{eqn:smooth-proj-fin-0}
{\underset{y \in X_{(1)}}\bigoplus} \ K^M_{i+1}(k(y))
\xrightarrow{\partial} {\underset{x \in X_{(0)}}\bigoplus} \ K^M_{i}(k(x))
\to \CH^{d+i}(X, i) \to 0.
\end{equation}
We thus get an exact sequence
\begin{equation}\label{eqn:smooth-proj-fin-2}
{\underset{y \in X_{(1)}}\bigoplus} \ {K^M_{i+1}(k(y))}/m
\xrightarrow{\partial} {\underset{x \in X_{(0)}}\bigoplus} \ {K^M_{i}(k(x))}/m
\to {\CH^{d+i}(X, i)}/m \to 0.
\end{equation}
If $k$ is algebraically closed, then $K^M_{i}(k)$ is divisible for $i \ge 1$
and hence we conclude that ${\CH^{d+i}(X, i)}/m = 0$ for $i \ge 1$.
This proves the lemma for $k$ algebraically closed.
For $k$ finite, the lemma follows from \cite[Thm.~1.2, 1.3]{Akthar}.

Suppose now that $k$ is a $p$-adic field with residue characteristic $p$. 
We first assume $i =1$. In this case, we can apply Temkin's 
strengthening of the alteration theorems of de Jong and Gabber 
(see \cite[Thm.~1.2.5]{Temkin}) to find a projective and generically
finite morphism $f \colon X' \to X$ of degree $p^n$ such that
$X'$ is a smooth projective scheme over $k$ which admits a strict semi-stable 
reduction over $k$. Since the composition $\CH^{d+1}(X, 1) 
\xrightarrow{f^*} \CH^{d+1}(X', 1) \xrightarrow{f_*} \CH^{d+1}(X, 1)$
is multiplication by $p^n$, it follows that
${\CH^{d+1}(X, 1)}/m \inj {\CH^{d+1}(X, 1)}/m$. We can therefore assume that
$X$ admits a strict semi-stable reduction over $k$. In this case, the
finiteness of ${\CH^{d+1}(X, 1)}/m$ follows from
\cite[Thm.~3.1.4]{Forre}.

We now assume $i \ge 2$.
Let $\sH^{p}_{\etl}(q)$ be the Zariski sheaf on $X$ associated
to the presheaf $U \mapsto H^p_{\etl}(U, \Lambda'(q))$.
Let $X_{(i)}$ be the set of points $x \in X$ such that $\dim(\ov{\{x\}}) = i$.
We consider the commutative diagram with exact rows:
\begin{equation}\label{eqn:smooth-proj-fin-3}
\xymatrix@C.4pc{
{\underset{y \in X_{(1)}}\bigoplus} \ {K^M_{i+1}(k(y))}/m
\ar[r]^-{\partial} \ar[d] & {\underset{x \in X_{(0)}}\bigoplus} \ 
{K^M_{i}(k(x))}/m \ar[r] \ar[d] &  {\CH^{d+i}(X, i)}/m \ar[r] \ar[d] & 0 \\
{\underset{y \in X_{(1)}}\bigoplus} \ H^{i+1}_{\etl}(k(y), \Lambda'(i+1)) 
\ar[r]^-{\partial} & {\underset{x \in X_{(0)}}\bigoplus} \ 
H^{i}_{\etl}(k(x), \Lambda'(i)) \ar[r] & H^d_\zar(X, \sH^{d+i}_{\etl}(d+i))
\ar[r] & 0,}
\end{equation}
where the bottom row is part of the Bloch-Ogus sequence for the
{\'e}tale cohomology sheaves and the first two vertical arrows are the
Norm-residue homomorphisms. 
The two vertical arrows from the left are isomorphisms by the
solution of the Bloch-Kato conjecture due to Rost and Voevodsky.
It follows that the right vertical arrow is also an isomorphism.
It suffices therefore to show that $H^d_\zar(X, \sH^{d+i}_{\etl}(d+i))$
is finite.

We now consider the Leray spectral sequence
\[
E^{p,q}_2 = H^p_\zar(X, \sH^{q}_{\etl}(d+i)) \Rightarrow
H^{p+q}_{\etl}(X, \Lambda'(d+i)).
\]
Since the {\'e}tale cohomological dimension of any affine in $X$ is bounded
by $d+2$, one easily checks that $E^{d,d+i}_2 = E^{d,d+i}_\infty$ and
there is an isomorphism
$E^{d,d+i}_\infty \xrightarrow{\cong} H^{2d+i}_{\etl}(X, \Lambda'(d+i))$.
But the latter group is finite by the proper base change theorem.
This concludes the proof.
\end{proof}

\begin{prop}\label{prop:Sing-proj-fin}
Let $k$ be as in \lemref{lem:smooth-proj-fin} and
let $X$ be any projective $k$-scheme of dimension $d \ge 0$ over 
$k$. Then $H^{2d+i}(X, \Lambda'(d+i))$ is finite for every $i \ge 0$.
\end{prop}
\begin{proof}
Let $q$ denote the exponential characteristic of $k$.
Since the localization map $H^{2d+i}(X, \Lambda'(d+i)) \to 
H^{2d+i}(X, \Lambda'(d+i))[\tfrac{1}{q}]$ is injective, it suffices to show that 
the latter group is finite. 
We can assume that $X$ is reduced.
If $X$ is smooth, we are done by \lemref{lem:smooth-proj-fin}.

If $X$ is not smooth, we argue as follows. By
Temkin's strengthening of the alteration theorems of de Jong and Gabber 
(see \cite[Thm.~1.2.5]{Temkin}),
there exists $W\in \Sm _k$ and a surjective proper map $h:W\rightarrow X$, 
which is generically finite and \'etale of degree
$q^n$, $n\geq 1$.  Then by a theorem of Raynaud-Gruson
\cite[Thm. 5.2.2]{Ray}, there exists a blow-up $g:X'\rightarrow X$ with 
center $Z$
such that the following diagram commutes, where $h'$ is finite, flat and surjective 
of degree $q^n$
and $g':W'\rightarrow W$ is the blow-up of W with center $h^{-1}(Z)$:
\begin{align}  \label{bwupsqr.1}
\begin{array}{c}
\xymatrix{W' \ar[d]_-{g'} \ar[r]^-{h'}& X' \ar[d]^-{g} \\
	W \ar[r]_-{h} & X.}
\end{array}
\end{align}

Thus we have a $cdh$-cover 
$\{X' \amalg Z \rightarrow X\}$ of $X$, such that $\dim_k(Z) < \dim_k(X)$ and 
$\dim_k(E) < \dim_k(X)$, where we set $E = X' \times _{X} Z$.
By $cdh$-excision of \cite[Prop.~5.3.4]{Kelly} (see also Thm.~2.3.14 of
the same reference), the following sequence is exact:
\begin{align*}
H^{2d+i-1}(E, \Lambda'(d+i))[\tfrac{1}{q}] 
\rightarrow 
H^{2d+i}(X, \Lambda'(d+i))[\tfrac{1}{q}] 
\rightarrow \hspace*{5cm} \\
\hspace*{6cm} H^{2d+i}(X', \Lambda'(d+i))[\tfrac{1}{q}] \oplus 
H^{2d+i}(Z, \Lambda'(d+i))[\tfrac{1}{q}].
\end{align*}
By induction on dimension, this reduces to showing that 
$H^{2d+i}(X', \Lambda'(d+i))[\tfrac{1}{q}]$ is finite.

If we let $Z' = h^{-1}(Z)$ and $E' = W' \times_W Z'$, then the same reasoning
as above shows that there is an exact sequence
\begin{align*}
H^{2d+i}(W, \Lambda'(d+i))[\tfrac{1}{q}] \to
H^{2d+i}(W', \Lambda'(d+i))[\tfrac{1}{q}] \to
\frac{H^{2d+i}(E', \Lambda'(d+i))}{H^{2d+i}(Z', \Lambda'(d+i))}[\tfrac{1}{q}].
\end{align*}
The left term is finite since $W$ is smooth and projective, and the right term
is finite by induction on the dimension. It follows that the middle term is
finite too. It suffices therefore to show that the map
$h^{\prime \ast} \colon H^{2d+i}(X',\Lambda'(d+i))[\tfrac{1}{q}] \rightarrow
H^{2d+i}(W',\Lambda'(d+i))[\tfrac{1}{q}]$ 
is injective. But this follows from \cite[Prop.~6.3]{Cisinski-Deglise}
because $h'$ is finite and flat (in particular, surjective) of degree $q^n$.
This completes the proof.
\end{proof}

\subsection{Proofs of parts (1) and (2) of \thmref{thm:Main-4}}
\label{sec:M4-pf-1}
The following result generalizes the finiteness theorems of 
Kato-Saito \cite{Kato-Saito-1} and Saito-Sato
\cite{SaitoSatoAnnals} to non-complete schemes.
The question whether such a finiteness result would hold originated out of
a discussion between Colliot-Th{\'e}l{\`e}ne and the authors.
In the two results below,  $k$ is either a finite, or a $p$-adic or an algebraically
closed field.

\begin{thm}\label{thm:smooth-fin}
  Let $U$ be a smooth quasi-projective $k$-scheme of pure dimension $d \ge 0$.
  Then ${H^{sus}_0(U)}/m$ is finite.
\end{thm}
\begin{proof}
We can assume that $U$ is integral.
Let $p$ be the exponential characteristic of $k$.
It suffices to show that $({H^{sus}_0(U)}/m)[\tfrac{1}{p}]$ is finite.
We have seen in \S~\ref{sec:LW-mot-coh} that
$({H^{sus}_0(U)}/m)[\tfrac{1}{p}]$ agrees with
$({H^{2d}_c(U, \Z(d))}/m)[\tfrac{1}{p}]$.
The theorem is thus equivalent to the statement that 
$({H^{2d}_c(U, \Z(d))}/m)[\tfrac{1}{p}]$ is finite.
By the universal coefficient theorem, we have an exact sequence
\begin{equation}\label{eqn:smooth-fin-0}
0 \to ({H^{2d}_c(U, \Z(d))}/m)[\tfrac{1}{p}]
\to H^{2d}_c(U, \Lambda'(d))[\tfrac{1}{p}] \to
~_mH^{2d+1}_c(U, \Z(d)))[\tfrac{1}{p}].
\end{equation}

We choose an open embedding $U \subset X$ such that $X$ is an integral
projective scheme of dimension $d$ over $k$. We let $D = X \setminus U$
with reduced closed subscheme structure.
We then have an exact sequence
\[
H^{2d}(D, \Z(d)))[\tfrac{1}{p}] \to H^{2d+1}_c(U, \Z(d)))[\tfrac{1}{p}] \to
H^{2d+1}(X, \Z(d)))[\tfrac{1}{p}].
\]
The two end terms are zero by \cite[Thm.~5.1]{KP}. It follows that so is the
middle term. Using ~\eqref{eqn:smooth-fin-0}, 
we are reduced to showing that
$H^{2d}_c(U, \Lambda'(d))[\tfrac{1}{p}]$ is finite.
But this follows from the exact sequence
\[
H^{2d-1}(D, \Lambda'(d))[\tfrac{1}{p}] \to 
H^{2d}_c(U, \Lambda'(d))[\tfrac{1}{p}] \to H^{2d}(X, \Lambda'(d))[\tfrac{1}{p}]
\]
and \propref{prop:Sing-proj-fin}.
\end{proof}

The second main result of this section is the following generalization of the
finiteness theorems of Kato-Saito \cite{Kato-Saito-1} and Saito-Sato
\cite{SaitoSatoAnnals} to Chow groups with modulus.

\begin{thm}\label{thm:Mod-fin}
  Let $X$ be a smooth projective $k$-scheme and $D \subset X$ an effective Cartier
  divisor. Then ${\CH_0(X|D)}/m$ is finite.
\end{thm}
\begin{proof}
If $k$ is finite, the theorem is already known by \cite[Thms.~1.3, 4.8]{GKris}. 
Next, suppose that $k$ is a $p$-adic field. We can then 
find a proper birational morphism $f \colon X' \to X$ from a smooth
projective scheme $X'$ such that $f^{-1}(D)_\red$ is a strict normal
crossing divisor on $X'$ and $f$ is an isomorphism over the complement of
$D$. It is clear from the definition of the Chow group of 0-cycles 
with modulus that the pull-back map $f^* \colon \CH_0(X|D) \to 
\CH_0(X'|D')$ exists and is an isomorphism, where $D' = f^*(D)$.
We can therefore assume that $D_\red$ is a simple normal crossing 
divisor. We are now done by Theorems~\ref{thm:Main-2-fin} and
~\ref{thm:smooth-fin}.

We now assume that $k$ is algebraically closed. Using ~\eqref{eqn:doubl-ex-seq},
we only need to show that $\CH^{BK}_0(S_X)_{\deg 0}$ is
$m$-divisible.  Since every 0-cycle of degree zero on $S_X$ lies on
a complete intersection reduced curve $C \subset S_X$, it suffices to
show that $\Pic^0(C)$ is $m$-divisible. But this is classical.
\end{proof}

\subsection{Proof of \thmref{thm:Main-4}(3)}\label{sec:FLW}
We shall now complete the proof of \thmref{thm:Main-4} by proving its part (3).
If $k$ is algebraically closed, the last part of the proof of \thmref{thm:Mod-fin}
shows that $\CH^{LW}_0(X)_{\deg 0}$ is $m$-divisible. If $k$ is finite,
part (3) of \thmref{thm:Main-4} follows from \cite[Thm.~1.3]{GK-1}.
The remaining case follows from the following general result of independent
interest.

Let $k$ be any field and let $m \neq 0$ be an integer.
Let $X$ be an integral projective $k$-scheme which is regular is codimension one.
It is shown in \cite[Lem.~8.3]{GK-1} that the identity map of
$\sZ_0(U)$ defines a surjection
$\theta_X \colon \CH^{LW}_0(X) \surj H^{sus}_0(U)$ if we let $U = X_\reg$.

\begin{thm}\label{thm:LW-Suslin}
Assume that ${\rm char}(k) = 0$. Then the map
  \[
    {\theta_X}/m \colon {\CH^{LW}_0(X)}/m \to {H^{sus}_0(U)}/m
  \]
  is an isomorphism.
\end{thm}
\begin{proof}
 There is nothing to prove if $d \le 1$ and we can thus assume that $d \ge 2$.
 We shall prove the statement by induction on $d$. We first consider the starting
 case $d = 2$.

Let $\pi \colon \wt{X} \to X$ be a resolution of singularities of 
$X$ such that the reduced exceptional divisor $E \subset \wt{X}$ has
strict normal crossings.  For any integer $n \ge 1$, let 
$nE \inj \wt{X}$ denote the infinitesimal thickening of $E$ in $\wt{X}$
of order $n$.

It is clear from the definitions of $\CH^{LW}_0(X)$,
$\CH_0(\wt{X}|D)$ and $H^{sus}_0(U)$ that the identity map of
$\sZ_0(U)$ defines, by the pull-back via $\pi^*$,
the canonical surjective maps
\begin{equation}\label{eqn:Surface-00}
\CH^{LW}_0(X) \stackrel{\pi^*}{\surj} \CH_0(\wt{X}|nE) \surj 
\CH_0(\wt{X}|E) \surj H^{sus}_0(U)
\end{equation}
for every integer $n \ge 1$ such that the composite map
is $\theta_X$.
The first arrow from the left is an isomorphism for all $n \gg 1$
by \cite[Thm.~1.5]{GK-1}, and the third arrow is an isomorphism modulo $m$
by \thmref{thm:Main-2-fin}.
We thus have to show that the kernel of the
surjection $\CH_0(\wt{X}|nE) \surj \CH_0(\wt{X}|E)$ is divisible.
By \cite[Thm.~8.2]{BKS}, this is equivalent to showing that $H^1_{\rm Zar}(E,
\sK_{2, (nE,E)})$ is divisible. But this is an immediate consequence of the well known
Goodwillie's theorem for the relative $K$-theory of nilpotent ideals,
which implies that $\sK_{2, (mE,E)}$ is a sheaf of $\Q$-vector spaces on $E$.

We now assume $d \ge 3$.
By \cite[Cor.~2.5]{GK-1}, we can assume that $X$ is normal.
We let $J^d(X)$ denote the Albanese variety of $X$ (which coincides with the
classical Albanese variety of any desingularization of $X$) and consider the
commutative diagram (see \thmref{thm:Semi-abl-Alb})
\begin{equation}\label{eqn:Surface-001}
  \xymatrix@C.8pc{
    \CH^{LW}_0(X)_{\deg 0} \ar@{->>}[r]^-{\theta_X} \ar[d]_-{\rho_X} &
    H^{sus}_0(U)_{\deg 0} \ar[d]^-{\rho_U} \\
    J^d(X) \ar@{->>}[r]^-{\theta'_X} & J^d(U).}
  \end{equation}
  The vertical arrows are isomorphisms on the torsion subgroups
  by \cite[Thm.~1]{Levine-AJM} and \cite[Thm.~1.1]{SS}.
  Since a surjective morphism of semi-abelian varieties induces surjection on
  $m$-torsion subgroups for every $m \ge 1$, it follows that the same happens
  for $\theta_X$. This implies that 
  $\Ker(\theta_X) \to \Ker({\theta_X}/m)$ is surjective.
  Hence, it suffices to show that the map
  $\Ker(\theta_X) \to {\CH^{LW}_0(X)}/m$ is zero.

Let $\nu \colon C \to X$ be a finite morphism from a regular integral
projective curve whose image is not contained in $X_{\rm sing}$
and let $f \in {\rm Ker}(\sO^{\times}_{C,E} \surj \sO^{\times}_{E})$,
where $E = \nu^{-1}(X_{\rm sing})$ with the reduced closed subscheme
structure. It suffices to show that $\nu_*(\divf(f))$ dies in ${\CH^{LW}_0(X)}/m$.
Equivalently, there is a 0-cycle $\alpha \in \sZ_0(U)$ such that
$\nu_*(\divf(f)) -m \alpha \in \sR^{LW}_0(X)$.

We can get a factorization $C \stackrel{\nu'}{\inj} 
\P^r_X \xrightarrow{\pi} X$ of $\nu$, where $\nu'$ is a closed
immersion and $\pi$ is the canonical projection. Since
the singular locus of $\P^r_X$ coincides with $\P^r_{X_{\rm sing}}$, it is
clear that $\nu'_*(\divf(f)) \in \sR^{S}_0(\P^r_{U})$ and
$\nu_*(\divf(f)) = \pi_*(\nu'_*(\divf(f)))$.
Using \cite[Cor.~2.6]{GK-1}, it suffices therefore to show that 
$\nu'_*(\divf(f)) - m \alpha' \in \sR^{LW}_0(\P^r_X)$ for some
$\alpha' \in \sZ_0(\P^r_U)$.
We can thus assume that $\nu \colon C \to X$ is a closed immersion.

We now fix a closed embedding $X \inj \P^N_k$.
Since $U$ and $C$ are smooth and
$d \ge 3$, we can use \cite[Thm.~7]{AK} or 
\cite[Thm.~3.9, Cor.~3.13]{GK} to find a hypersurface $H \subset \P^N_k$
containing $C$ and not containing $X$ 
such that the scheme theoretic intersection $X \cap H$
has the property that it contains
no irreducible component of $X_\sing$ and $H \cap U$ is smooth and integral.

We let $Y = H \cap U$ and $W = H \cap U$.
Since $Y$ is a Cartier divisor on $X$, we must have $Y_\sing = X_\sing \cap Y$.
Furthermore, $\dim(Y_\sing) =
\dim(Y \cap X_\sing) \le \dim(X_\sing) - 1 \le d-3 = \dim(Y) - 2$. 
In particular, $Y$ is regular in codimension one.
Let $C \stackrel{\nu'}{\inj} Y \stackrel{\iota}{\inj} X$ 
be the factorization of $\nu$.
It follows from the choice of $Y$ that 
$\nu'_*(\divf(f)) \in \sR^{S}_0(W)$. It follows by induction on $d$ that 
$\nu'_*(\divf(f))$ dies in ${\CH^{LW}_0(Y)}/m$.
The push-forward map $\iota_* \colon \sZ_0(W) \to \sZ_0(U)$
and \cite[Cor.~2.6]{GK-1} together imply that
$\nu_*(\divf(f))$ dies in ${\CH^{LW}_0(X)}/m$. This concludes the proof.
\end{proof}

\section{Chow groups of normal crossing surfaces}
\label{sec:Surface-00}
We are now left with proving \thmref{thm:Main-6}, the last of the main results
listed in \S~\ref{sec:Introduction}.
In order to do this, we need to use an induction procedure on the dimension of
normal crossing varieties.
In this section, we shall prove some results for surfaces which will
allow this induction.

\subsection{Chow group with modulus to 
Levine-Weibel Chow group}\label{sec:CMLW}
The goal of this subsection is to prove the $d =2$ case of \thmref{thm:Reg-mod-LW}.
Let $k$ be a field and $X$ a normal crossing variety of dimension $d \ge 1$ over $k$
(see Definition~\ref{defn:ncg}).
We shall use the following notations in this section.

Let $\iota \colon Y \inj X$ be the inclusion of an irreducible component 
of $X$ and let $E \subset Y$ be the scheme theoretic intersection
$\ov{X \setminus Y} \cap Y$. Then $Y$ is a regular 
and $E$ is a simple normal crossing
divisor on $Y$. Let $D \subset X$ denote the singular locus of $X$.
Let $\iota' \colon Y' = \ov{X \setminus Y} \inj X$ be the inclusion of the
union of other components of $X$.
Let $\sI \subset \sO_X$ be the sheaf of ideals on $X$ defining $Y'$
and let $\sJ \subset \sO_Y$ be the sheaf of ideals on $Y$ defining $E$.

We have the following canonical inclusion of the groups of 0-cycles:
\begin{equation}\label{eqn:0-cycle-incl}
\iota_* \colon \sZ_0(Y, E) \inj \sZ_0(X, D).
\end{equation}
We wish to show that this map preserves the subgroups
of rational equivalences when $d \le 2$. This will be generalized 
to higher dimensions in the next section. We first deal with the case $d =1$.

\begin{lem}\label{lem:NC-surface-3}
$\iota_*$ descends to a homomorphism
$\iota_* \colon \CH_0(Y|E) \to \CH^{LW}_0(X)$ when $d = 1$.
\end{lem}
\begin{proof}
Let $f \in {\rm Ker}(\sO^{\times}_{Y, E} \to
\sO^{\times}_{E})$. 
The scheme $X$ being a normal crossing variety means that the
sequence
\begin{equation}\label{eqn:NC-surface-3-0}
0 \to \sO_{X, E} \to \sO_{Y,E} \times \sO_{Y',E} \to \sO_{E} \to 0
\end{equation}
is exact.
If we let $f' = (f,1) \in  \sO^{\times}_{Y,E} \times \sO^{\times}_{Y',E}$, then
the modulus condition for $f$ implies that $f' \in \sO^{\times}_{X,E}$.
Moreover, we have $\divf(f') = \divf(f)$ in $\sZ_0(X,D)$.
\end{proof}

In the rest of \S~\ref{sec:CMLW}, we shall assume $d = 2$.
For any closed immersion of schemes $W_2 \subset W_1$ of $k$-schemes, we let 
$\sK_{i, (W_1, W_2)}$ denote the Zariski sheaf on $W_1$ associated to the
presheaf $U \mapsto K_i(U, U \cap W_2)$ of relative Quillen $K$-theory.

\begin{lem}\label{lem:ideal}
The canonical map
$\iota^* \colon \sI \to \iota_*(\sJ)$ is an isomorphism.
\end{lem}
\begin{proof}
Let $A$ be a reduced
affine $k$-algebra of dimension $d \ge 1$ such that
$\Spec(A)$ is a normal crossing variety.
Let $\{\fp_1, \ldots , \fp_r\}$ be the set of minimal primes of $A$.
We write $I = \stackrel{r}{\underset{i = 2}\bigcap} \ \fp_i$ and 
$B = {A}/{\fp_1}$. Our problem reduces to showing that the homomorphism
$I \to IB$ is bijective. This homomorphism is clearly surjective.
It is also injective because the kernel of the composite map
$I \to IB \to B$ is $\fp_1 \cap I = \stackrel{r}{\underset{i = 1}\bigcap}  
\fp_i = (0)$.
\end{proof}

\begin{lem}\label{lem:Surface-K-0}
The restriction
map $\iota^*\colon H^2(X, \sK_{2, (X, Y')}) \to H^2(Y, \sK_{2, (Y, E)})$ is an 
isomorphism.
\end{lem}
\begin{proof}
We have a commutative diagram of Zariski sheaves on $X$ with exact rows:
\begin{equation}\label{eqn:Surface-K-0-0}
\xymatrix@C1pc{
\sK_{2, (X, Y')} \ar[r] \ar[d]_{\iota^*} & \sK_{2,X} \ar[r] \ar[d]^{\iota^*} &
\iota'_*(\sK_{2,Y'}) \ar[d]^{\iota^*} \ar[r] & 0 \\
\sK_{2, (Y, E)} \ar[r] & \sK_{2,Y} \ar[r] &
\iota'_*(\sK_{2,E}) \ar[r] & 0.} \\ 
\end{equation}

The canonical map
$\iota^* \colon \sI \to \iota_*(\sJ)$ is an isomorphism by \lemref{lem:ideal}. 
It follows from \cite{Keune} (see also \cite[Thm.~0.2]{GW}) that 
the map $\sK_{2, (X,Y')} \to \iota_*(\sK_{2,(Y,E)})$ is surjective.
It is also clear that the kernel of this map is supported on $D$.
In particular, the map $H^2(X,\sK_{2, (X,Y')}) \to 
H^2(X, \iota_*(\sK_{2,(Y,E)})) \simeq H^2(Y, \sK_{2,(Y,E)})$ is an isomorphism. 
\end{proof}

For any reduced quasi-projective surface $Z$ over $k$, let 
\[cyc_Z\colon \CH^{LW}_0(Z) \to F^2K_0(Z)\subset K_0(Z)\]
be the $K$-theoretic cycle class map (e.g., see \cite[Lem.~3.12]{BK}). 
As we have already observed, it factors through $H^2_{\rm Zar}(Z, \sK_{2,Z})$, and we
get canonical maps
\begin{equation}\label{eqn:cyc-cl-map}
  \CH^{LW}_0(Z) \to \CH_0^{BK}(Z) \to
  H^2_{\rm Zar}(Z, \sK_{2,Z}) \to H^2_{\rm Nis}(Z, \sK_{2,Z}).
  \end{equation}
The second arrow is an isomorphism by \cite[Thm.~8.1]{BKS}
(or by \cite{LevineBlochFormula} when $k$ is algebraically closed),
the third arrow is an isomorphism by \cite[Thm.~2.5]{Kato-Saito-2}
(the cited result assumes $X_\reg$ to be smooth but \cite[\S~3.5]{GKris2}
explains that this is unnecessary). The first arrow is an isomorphism when
$k$ is infinite by \cite[Cor.~7.8]{BKS}.

\begin{lem}\label{lem:NC-surface-4}
The map $\iota_*$ of 
~\eqref{eqn:0-cycle-incl} descends to a
homomorphism
\[
\iota_*\colon \CH_0(Y|E) \to \CH^{BK}_0(X).
\] 
\end{lem}
\begin{proof}
Since $d = 2$ and $H^i(X, \iota'_*(\sK_{2,Y'})) \xrightarrow{\simeq}
H^i(Y', \sK_{2,Y'})$ for $i \ge 0$, we get from ~\eqref{eqn:Surface-K-0-0},
a commutative diagram of Zariski cohomology groups:
\begin{equation}\label{eqn:Surface-K-0-1}
\xymatrix@C1pc{
H^2(X, \sK_{2, (X, Y')}) \ar[r] \ar[d]^{\simeq}_{\iota^*} & 
H^2(X, \sK_{2,X}) \ar[r] 
\ar[d]^{\iota^*} &
H^2(Y', \sK_{2,Y'}) \ar[d]^{\iota^*} \ar[r] & 0 \\
H^2(Y, \sK_{2, (Y, E)}) \ar[r] & H^2(Y, \sK_{2,Y}) \ar[r] &
H^2(E, \sK_{2,E}) \ar[r] & 0,} \\ 
\end{equation}
where the two rows are exact.

For a closed point $x \in Y \setminus E = X \setminus Y' \subset X_{\rm reg}$,
we have a commutative diagram
\begin{equation}\label{eqn:Surface-K-0-2}
\xymatrix@C.8pc{
\Z \ar[r]^-{\simeq} & H^2_{\{x\}}(X, \sK_{2, (X, Y')}) \ar[rr] \ar[dd]_{\simeq} 
\ar[dr]^{\simeq} & & H^2(X, \sK_{2, (X, Y')}) \ar[dr] \ar[dd] \\
& & H^2_{\{x\}}(X, \sK_{2, X}) \ar[rr] \ar[dd]_>>>>>>{\simeq} & &  
H^2(X, \sK_{2, X}) 
\ar[dd]^{\iota^*} \\
& H^2_{\{x\}}(Y, \sK_{2, (Y, E)}) \ar[rr] \ar[dr]_{\simeq} & & H^2(Y, \sK_{2, (Y, E)})
\ar[dr] & \\
& & H^2_{\{x\}}(Y, \sK_{2, Y}) \ar[rr] & & H^2(Y, \sK_{2, Y}).} 
\end{equation}

Since the composite $\Z \xrightarrow{\simeq} H^2_{\{x\}}(X, \sK_{2, X})
\inj K_0(X)$ is the cycle class of $x \in X_{\rm reg}$, it follows from
\lemref{lem:Surface-K-0}, ~\eqref{eqn:Surface-K-0-1} and 
~\eqref{eqn:Surface-K-0-2} that the inclusion $Y \setminus E = X \setminus Y' 
\subset X_{\rm reg}$ induces a commutative diagram of cycle class maps
\begin{equation}\label{eqn:Surface-K-0-3}
\xymatrix@C.8pc{
\sZ_0(Y, E) \ar[r]^-{cyc_{Y|E}} \ar@{^{(}->}[d]_-{\iota_*} & H^2(Y, \sK_{2, (Y, E)}) 
\ar[d]^-{\phi_Y}  & \\
\sZ_0(X,D) \ar[r]^-{\cyc_X} & H^2_{\rm Zar}(X, \sK_{2,X}) \ar@{^{(}->}[r] &  
K_0(X),}
\end{equation}
where $\phi_Y = (\iota^*)^{-1} \colon H^2(Y, \sK_{2, (Y, E)}) \to 
H^2(X, \sK_{2, (X, Y')})$.

It follows from \cite[Thm.~3.6]{BKS} that $cyc_X$ factors
through the quotient $\CH^{BK}_0(X)$ and it follows from 
\cite[Thm.~1.2, Lem.~2.1]{Krishna-3} that the map $cyc_{Y|E}$
factors through the quotient $\CH_0(Y|E)$. We thus get a commutative
diagram with solid arrows
\begin{equation}\label{eqn:Surface-K-0-4}
\xymatrix@C.6pc{
  \sZ_0(Y, E) \ar@{->>}[r] \ar@{^{(}->}[d]_-{\iota_*} &
  \CH_0(Y|E) \ar[r]^-{cyc_{Y|E}} &
H^2_{\rm Zar}(Y, \sK_{2, (Y, E)}) \ar[d]^-{\phi_Y} & \\
\sZ_0(X,D) \ar@{->>}[r] & \CH^{BK}_0(X) \ar[r]^-{cyc_X} &  H^2_{\rm Zar}(X, \sK_{2,X})
\ar@{^{(}->}[r] & K_0(X).}
\end{equation}

Now, the map $\cyc_X\colon \CH^{BK}_0(X) \to  H^2_{\rm Zar}(X, \sK_{2,X})$  is an 
isomorphism as shown above, and the map $ \CH_0(Y|E) \to 
H^2_{\rm Zar}(Y, \sK_{2, (Y, E)})$ is an isomorphism by \cite[Thm.~8.2]{BKS}. It 
follows that the composite map
$\iota_* \colon \sZ_0(Y, E)  \inj \sZ_0(X,D) \surj \CH^{BK}_0(X)$ factors through
$\CH_0(Y|E)$. This proves the lemma.
\end{proof}

\subsection{0-cycles on normal crossing surfaces}
\label{sec:0-cyc-surf}
In this subsection, we shall prove the $d =2$ case of \thmref{thm:EKW-LW}.
We continue with the setup of \S~\ref{sec:CMLW}.
Let $\{X_1, \ldots , X_r\}$ be the set of irreducible components of $X$. 
Let $C \subset X$ be a reduced curve. The following definition 
is borrowed from \cite{EKW}.

\begin{defn}\label{defn:snc-curve}
We shall say that $C$ is a {\sl snc subcurve} in $X$ if its scheme theoretic
intersection with each irreducible component of $X$ is either empty or smooth and
integral of dimension one, its intersection with each $X_i \cap X_j$ ($i \neq j$)  
is either empty or reduced and 0-dimensional, and its intersection with each
$X_i \cap X_j \cap X_{\ell}$ ($i \neq j \neq \ell \neq i$) is empty.
\end{defn}

Let $X$ be as above and let $C \subset X$ be a snc subcurve.
The following is easy to verify.
\begin{lem}\label{lem:snc-curve-reg}
The following hold for $C$.
\begin{enumerate}
\item
The inclusion $C \inj X$ is a regular embedding.
\item
The embedding dimension of $C$ is at most two.
\item
$C$ is regular and integral 
if and only if it is contained in one component $X_i$ of
$X$ and $C \cap X_j = \emptyset$ for every component $X_j \neq X_i$.
\end{enumerate}
\end{lem}

Let $X$ be a normal crossing variety of dimension $d \ge 2$ over $k$.
Let $D \subset X$ denote the singular locus of $X$. We let 
$\sR^{snc}_0(X,D)$ denote the subgroup of $\sZ_0(X,D)$ generated by
$\divf(f)$, where $f$ is a rational function on a snc subcurve
$C \subset X$ such that $f$ is regular and invertible along $C \cap D$.
We let 
\begin{equation}\label{eqn:snc-Chow}
\CH^{snc}_0(X) = \frac{\sZ_0(X,D)}{\sR^{snc}_0(X,D)}.
\end{equation}

It follows from \lemref{lem:snc-curve-reg} that a snc subcurve on $X$ is
a Cartier curve relative to $X_{\rm sing}$. This implies that there are
canonical surjections $\CH^{snc}_0(X) \surj \CH^{LW}_0(X) \surj \CH_0^{BK}(X)$.
We will now prove the following result.
This shows that the Levine-Weibel Chow group of a normal crossing surface
can be described using a simpler notion of rational equivalence. 
This will be used in the proof of \thmref{thm:EKW-LW}.

\begin{prop}\label{prop:snc-surface}
Assume that $k$ is infinite and $d =2$, then the maps 
\[
\CH^{snc}_0(X) \surj \CH^{LW}_0(X) \surj \CH_0^{BK}(X)
\]
are isomorphisms.
\end{prop}
\begin{proof}
We have remarked previously that the second arrow is an isomorphism.
We shall show that the first arrow is injective using the ideas of
\cite[Lem.~5.5]{BK}, that in turn is inspired by  \cite[Lem.~1.4]{Levine-2}.
We fix a locally closed embedding $X \inj \P^n_k$.
Let $C \subset X$ be a reduced Cartier curve relative to $D$ and
let $f \in k(C)^{\times}$ be such that it is regular and invertible along 
$C \cap D$. We need to show that $\divf(f) \in \sR^{snc}_0(X,D)$.

As in the proof of \cite[Lem.~1.4]{Levine-2}, we can reduce to the case
when there is a very ample line bundle $\sL$ on $X$ and a section
$s \in H^0(X, \sL)$ such that $C = (t_0)$. Since $\sL$ is very ample,
the  Bertini theorem (see the proof of \cite[Lem.~1.3]{Levine-2}
and \cite[Chap.~6]{Jou}, 
and also \cite{GK} for the case where $k$ is not assumed to be perfect)  
now allow us to choose another Cartier divisor $C_{\infty}$ in the linear 
system $H^0(X, \sL)$ such that:
\begin{enumerate}
\item
$C_{\infty} \subset X$ is a snc subcurve.
\item
$C_{\infty} \cap C \cap D = \emptyset$.
\item
$C_{\infty}$ contains no component of $C$.
\end{enumerate}

Denote by $t_{\infty}$ the section of $\sL$ with  $C_\infty = (t_{\infty})$. 
We extend the function $f$ on $C$ to a function $h$ on 
$C_{\infty} \cup C$ by setting $h=f$ on $C$ and $h=1$ on $C_{\infty}$. Notice 
that $h$ is meromorphic on $C_{\infty} \cup C$ and regular
invertible in a neighborhood of $(C_{\infty} \cup C) \cap D$ by (2). Let $S$ 
denote the finite set of points 
$S = \{\text{poles of $f$}\} \cup (C\cap C^{\infty})$. 
Note that $S\cap D = \emptyset$.

We now choose a very ample line bundle $\sM = j^*(\sO_{\P^r_k}(1))$ on $X$, 
corresponding to an embedding $j\colon X \inj \P^r_k$ for $r \gg 0$ of $X$ as 
a locally closed subscheme. Since $k$ is infinite, there is a dense open 
subset $V_X$ of the dual projective space $(\P_k^r)^\vee$ such that for 
$L \in V_X$, the scheme theoretic intersection $L\cdot X = L\times_{\P^r_k} X$ 
satisfies the following list of properties:
\begin{listabc}
	\item $L\cdot X$ is a snc subcurve in $X$,
        \item $L\cdot X \cap (C\cup C_\infty) \cap D = \emptyset$,
	\item $L\cdot X\supset S$,
    \item  $L\cdot Y$ intersects $C_{\infty} \cup C$ in a finite set of points.
\end{listabc}

The hyperplane $L$ corresponds to a section $l$ of the linear system 
$H^0(\P^r_k, \sO_{\P^r}(1))$. Write $s_\infty$ for the global section of 
$\sL$ that is the restriction of $l$ to $X$. Then $L\cdot X = (s_\infty)$. 
Note that we can assume that $s_\infty$ does not have poles on 
$(C\cup C_\infty) \cap D$. 
Let $\ov{X}$ be the closure of $X$ in $\P^r_k$ and let $\sI$ be the ideal 
sheaf of $\ov{X} \setminus X$ in $\P^r_k$.  We can find a section 
$s'_{\infty}$ of the sheaf $\sI \otimes \sO_{\P^N_k}(m)$ for some
$m \gg 0$, which restricts to a section $s_{\infty}$ on $X$ satisfying
the properties (a) - (d) on $X = \ov{X} \setminus V(\sI)$.
This implies in particular that 
$X \setminus (s_{\infty}) = \ov{X} \setminus (s'_{\infty})$ 
is affine. Thus, up to taking a further Veronese embedding of $\P^r_k$ 
(replacing $s_\infty$ with $s_\infty'$) we can assume that (a) - (d) as above as 
well as the following hold:
\\
\hspace*{1cm} e) \ $X \setminus (s_\infty) = X \setminus (L\cdot X)$ is affine.

Consider again the function $h$ on $C \cup C_\infty$. By our choice of $S$, 
$h$ is regular on $(C\cup C_\infty) \setminus S$. By (e) above, $h$ extends to a 
regular function $H$ on the affine open $U = X \setminus (s_{\infty})$. Since 
$H$ is a meromorphic function on $X$ which has poles only along
$(s_{\infty})$, it follows that for $N\gg 0$ the section $H s_\infty^N$ is an 
element of $H^0(U, \sL^{N})$ which extends to a section $s_0$ of $\sL^N$ 
on all of $X$. Since $h$ is regular and invertible at each point of 
$C\cup C_\infty \cap D$ and since $s_\infty$ does not have zeros or poles on 
$C\cup C_\infty \cap D$, it follows that 
$(s_0)\cap (C\cup C_\infty)\cap D = \emptyset$ and (using (d) above) that 
$(s_0)$ does not contain any component of $C\cup C_\infty$. Note that up to 
replacing $s_0$ by $s_0s_\infty^i$, we are free to choose $N$ as large as 
needed.

Write $\sI_{C \cup C_\infty}$ for the ideal sheaf of $C\cup C_\infty$ in $X$. We 
can then find sections $s_1,\ldots, s_m$ of 
$H^0(X, \sL^{N} \otimes \sI_{C \cup C_\infty})$ such that the rational map 
$\phi\colon X\dashrightarrow \P^{m-1}_k$ that they define is a locally closed 
immersion on $X\setminus (C\cup C_\infty)$. In particular, there exists an 
affine open neighbourhood $U_x$ of every 
$x\in X\setminus (C\cup C_\infty)$ where at least one of the $s_i$ is not 
identically zero and where the $k$-algebra $k[s_1/s_i, \ldots, s_m/s_i]$ 
generated by $s_1/s_i, \ldots, s_m/s_i$ coincides with the coordinate ring of 
$U_x$. But then, the same must be true for the algebra  
$k[s_0/s_1, s_1/s_i, \ldots, s_m/s_i]$ obtained by adding the element 
$s_0/s_i$. Hence, the rational map $\psi\colon X \dashrightarrow \P^m_k$ given 
by the sections $(s_0, s_1, \ldots, s_m)$ of $H^0(X, \sL^{N})$ is also a 
locally closed immersion on $X\setminus (C\cup C_\infty)$, and since the base 
locus of the linear system associated to $(s_0, s_1, \ldots, s_m)$  is 
$(s_0)\cap (C\cup C_\infty)$, it is in fact a morphism away from 
$(s_0) \cap (C\cup C_\infty)$.

In particular, $\psi$ is birational (hence separable) to its image so that the linear 
system $V = (s_0, \ldots, s_m)$ is not composite with a pencil. 
Since $X$ is a simple normal crossing divisor, the 
classical Theorem of Bertini (see, for example, \cite[Thm.~I.6.3]{Zar58}
and \cite[Chap.~6]{Jou}) 
shows that a general divisor $E$ in $V$ is a snc subcurve in $X$.
We can therefore assume that there is a global section $s_0'$ of 
$H^0(X, \sL^{N})$ of the form $s_0' = s_0 +\alpha$ with 
$\alpha \in (s_1,\ldots, s_m)$, that satisfies the following properties:
\begin{listabcprime}
\item 
$(s'_0)$ is a snc subcurve in $X$.
\item
$(s'_0) \cap (C \cup C_{\infty}) \cap D = \emptyset$.
\item
$C_{\infty} \cup C$ contains no component of $(s'_0)$.
\end{listabcprime}

We then have 
\[
\frac{s'_0}{s^N_{\infty}} = \frac{Hs^N_{\infty} + (\alpha s^{-N}_{\infty})
s^N_{\infty}}{s^N_{\infty}} = H + \alpha s^{-N}_{\infty} = H', \ (\mbox{say}).
\]

Since $\alpha$ vanishes along $C \cup C_{\infty}$ and $s_{\infty}$ does not 
vanish identically on $U\cap (C\cup C_\infty)$ by (d), it follows that 
$H'_{|{(C \cup C_{\infty}) \cap U}} =
H_{|{(C \cup C_{\infty}) \cap U}} = h_{|U}$. In other words, we have
${s'_0}/{s^N_{\infty}} = h$ as rational functions on $C \cup C_{\infty}$.
We can now compute:

\[
\nu_*({\rm div}(f)) = (s'_0) \cdot C - N (s_\infty) \cdot C
\]
\[ 
0 = {\rm div}(1) = (s'_0) \cdot C_{\infty} - N (s_\infty)\cdot C_{\infty}.
\]

We therefore have
\[
\begin{array}{lll}
\nu_*(\divf(f)) & = & (s'_0) \cdot (C - C_{\infty}) - 
N(s_{\infty})(C - C_{\infty}) \\
& = & (s'_0) \cdot (\divf({t_0}/{t_{\infty}})) - N(s_{\infty}) \cdot
(\divf({t_0}/{t_{\infty}})) \\
& = & \iota_{{s'_0}, *}(\divf(f')) - N \iota_{{s_{\infty}}, *}(\divf(f'')),
\end{array}
\]
where $f' = ({t_0}/{t_{\infty}})|_{(s'_0)} \in 
\sO^{\times}_{(s'_0), D \cap (s'_0)}$ (by (b')) and 
$f'' = (t_0/t_{\infty})|_{(s_\infty)} \in 
\sO_{(s_\infty), D\cap (s_\infty) }^{\times}$ (by (b)). 
Since $(s'_0)$ and $(s_\infty)$ are snc subcurves on $X$, we are done.
\end{proof}

\section{Proof of \thmref{thm:Main-6}}
\label{sec:high-dim}
In this section, we study groups of 0-cycles on higher
dimensional normal crossing varieties and prove \thmref{thm:Main-6}.
We begin with a moving lemma from \cite{BK} and \cite{GKR}.

\subsection{A moving lemma}\label{sec:ML}
Let $k$ be a perfect field and $X$ an integral smooth projective
$k$-scheme with an effective Cartier divisor $D \subset X$.
We shall use the following condition in some results of this section.

\begin{equation}\label{eqn:star}
  \mbox{The canonical map $\CH^{LW}_0(S_X) \surj \CH_0^{BK}(S_X)$ is an isomorphism}.
  \hspace*{2cm}
\end{equation}

\vskip .2cm

The question as to when this condition is satisfied is a subtle one.
We summarize here some known facts about it. 
\begin{lem} \label{lem:comparisonLWlci}
The condition ~\eqref{eqn:star} holds in the following cases.
\begin{enumerate}
    \item $X$ is affine and $k$ is algebraically closed.
    \item ${\rm char}(k)=0$ and $X$ is projective.
    \item $\dim(X)\leq 2$ and $k$ is infinite.
    \item $k$ is algebraically closed and $D$ is reduced.
\end{enumerate}
\begin{proof}
Items (1) and (2) are respectively, the items (2) and (3) of
\cite[Thm.~3.17]{BK}. The item (2) is \cite[Cor.~7.8]{BKS} and
(4) follows from \cite[Thm.~6.6]{Krishna-2}.
\end{proof}
\end{lem}

Let $A \subset D$ be a closed subscheme of dimension 
at most $\dim(X) - 2$.
Let $\sR^{\rm mod}_0(X|D) \subset \sZ_0(X, D)$ be the subgroup generated by
$\divf(f)$, where $C \subset X$ is an integral curve
and $f \in k(C)^{\times}$ is such that the following hold.
\begin{enumerate}
\item
$C \not\subset D$.
\item
$C$ is regular along $D$.
\item
$C \cap A = \emptyset$.
\item
$f \in {\rm Ker}(\sO^{\times}_{C,\nu^*(D)} \surj \sO^{\times}_{\nu^*(D)})$, where
$\nu: C \inj X$ is the inclusion.
\end{enumerate}

Let $\CH^{\rm mod}_0(X|D)$ be the quotient of $\sZ_0(X, D)$ by 
$\sR^{\rm mod}_0(X|D)$. There is an evident surjection
$\CH^{\rm mod}_0(X|D) \surj \CH_0(X|D)$.

\begin{lem}\label{lem:reg-mod}
  Assume that $k$ is infinite and the condition ~\eqref{eqn:star} holds. Then the map
\[
\CH^{\rm mod}_0(X|D) \surj \CH_0(X|D)
\]
is an isomorphism.
\end{lem}
\begin{proof}
This is a direct consequence of \cite[Prop.~4.3]{GKR}
(which is \cite[Cor.~7.2]{BK} when $A = \emptyset$).
By  \cite[Prop.~4.3]{GKR}, there are maps $\tau^*_X \colon
\CH^{LW}_0(S_X) \to \CH^{\rm mod}_0(X|D)$ and $p_{+,*} \colon \CH^{\rm mod}_0(X|D) \to
\CH^{LW}_0(S_X)$ such that $\tau^*_X \circ p_{+,*}$ is identity.
By \cite[Thm.~7.1]{BK}, there are maps
$\tau^*_X \colon
\CH^{BK}_0(S_X) \to \CH_0(X|D)$ and $p_{+,*} \colon \CH_0(X|D) \to
\CH^{BK}_0(S_X)$ such that $\tau^*_X \circ p_{+,*}$ is identity.
In particular, the horizontal arrows in the commutative diagram
\begin{equation}\label{eqn:reg-mod-0}
  \xymatrix@C.8pc{
\CH^{\rm mod}_0(X|D) \ar[r]^-{p_{+,*}} \ar@{->>}[d]_-{can} & 
\CH^{LW}_0(S_X) \ar@{->>}[d]^{can} \\
\CH_0(X|D) \ar[r]^-{p_{+,*}} & \CH_0^{BK}(S_X)}
\end{equation}
are (split) injective.
The right vertical arrow is an isomorphism under ~\eqref{eqn:star}.
The lemma now follows.  
\end{proof}

\subsection{Generalization of \lemref{lem:NC-surface-4} in higher 
dimensions}\label{sec:Spl-fiber}
Let $k$ be a perfect field and $X$ a normal crossing variety of dimension
$d \ge 1$ over $k$ (see Definition~\ref{defn:ncg}). 
We shall now generalize
\lemref{lem:NC-surface-4} to higher dimensions. We shall use
the notations of \S~\ref{sec:CMLW} in the proof.

Recall that for a $k$-scheme $Y$ and a point $y \in Y$, the embedding
dimension of $Y$ at $y$ is $\edim_y(Y) =
\dim_{k(y)}(\Omega^1_{Y/k}\otimes _{\sO_Y} k(y))$,
where $\Omega^1_{Y/k}$ is the Zariski sheaf of K{\"a}hler differentials on $Y$.
For any integer $e \ge 0$, we let $Y_e$ denote the subscheme
of points $y \in Y$ such that ${\rm edim}_y(Y) = e$.
Since $y \mapsto {\rm edim}_y(Y)$ is an upper semi-continuous function 
on $Y$ (see \cite[Example~III.12.7.2]{Hartshorne}), 
it follows that $Y_e$ is a locally closed subscheme of $Y$. 
We let ${\rm edim}(Y) = {\underset{e \ge 0}{\rm max}}
\{e + \dim(Y_e)\}$.

\begin{thm}\label{thm:Reg-mod-LW}
  Assume that the condition ~\eqref{eqn:star} holds
  (e.g., $k$ is algebraically closed and 
$D$ is reduced, or  $k$ is infinite and $d \leq 2$).
Then the map $\iota_* \colon \sZ_0(Y, E) \inj \sZ_0(X, D)$ of 
~\eqref{eqn:0-cycle-incl} descends to a group 
homomorphism 
\[
\iota_* \colon \CH_0(Y|E) \to \CH^{LW}_0(X).
\]
\end{thm}
\begin{proof}
We shall prove the theorem by induction on $d$.
The case $d \le 2$ follows from Lemmas~\ref{lem:NC-surface-3} and
~\ref{lem:NC-surface-4} (and in this case the result is unconditional). 
So we assume $d \ge 3$. 
Let $C \subset Y$ be an integral curve not contained in $E$
and let $f \in {\rm Ker}(\sO^{\times}_{C^N,\nu^*(E)} \to \sO^{\times}_{\nu^*(E)})$,
where $\nu: C^N \to X$ is the canonical map.
We need to show that $\divf(f) \in \sR^{LW}_0(X,D)$.

By \lemref{lem:reg-mod}, we can assume that $C$ is regular in a
neighborhood of $C \cap D$ so that $\nu:C^N \to X$ is a closed embedding
in a neighborhood of $C \cap D$ and $C \cap D_{\rm sing} = C \cap E_{\rm sing} = 
\emptyset$. Note here that $E \subset Y$ is a simple normal crossing divisor.
In particular, it is reduced.

Let $T$ be the set of singular points of $C$. Since $C$ is 
regular along $D$, we must have $T \cap Z = \emptyset$,
where we let $Z = \ov{X \setminus Y}$.
We can find a finite sequence of blow-ups $\pi: X' \to X$ over the 
points of $T$ such that the strict transform $C'$ of $C$ is regular and
it intersects the exceptional divisor $F$ transversely in the regular locus of
$F$. As $T \cap Z = \emptyset$, the map $\pi^{-1}(Z) \to Z$ is an isomorphism.
In particular, $X'$ is a normal crossing variety of dimension $d$ over $k$.
Let $Y'$ be the strict transform of $Y$ under $\pi$
and let $E' = F + E$.
Then $E'$ is a simple normal crossing divisor on $Y'$ and $C'$ intersects
$E'$ transversely in its regular locus.
Note also that $D':= X'_\sing \cong D$.

We now choose a locally closed embedding $X' \inj \P^N_k$. 
We then have:
\begin{enumerate}
\item
${\rm edim}(C') = 2 < \dim(Y')$.
\item
${\rm edim}(C' \cap Y_i) + \dim(C \cap Y_i) \le 1 < \dim(Y_i)$ for every irreducible
component $Y_i \neq Y'$ of $X'$.
\item
${\rm edim}(C' \cap D_i) + \dim(C' \cap D_i) \le 1 < \dim(D_i)$ for every irreducible
component $D_i$ of $D'$. 
\item
$C' \cap D'_{\rm sing} = \emptyset$.
\end{enumerate}

Let $\ov{C'}$ denote the scheme theoretic closure of $C'$ in $\P^N_k$ and let
$\sI$ be the sheaf of ideals on $\P^N_k$ defining $\ov{C'}$.
Let $\{X_1, \ldots , X_r\}$ be the set of irreducible components of $X$
with $X_1 = Y$. Then $\{Y', X_2, \ldots , X_r\}$ is the set of irreducible
components of $X'$.
Let $\sW$ be the collection of all subschemes of $X'$ such that
$W \subset X'$ lies in $\sW$ if and only if it is an irreducible component
of $X'_J$ for some $J \subset [1,n]$ (see \S~\ref{sec:CMLW}).
Note that any such $W$ is integral and smooth.

If $k$ is infinite, we can use (1) - (4) and \cite[Thm.~3.12]{GK} to  
find $d \gg 0$ such that a general member $H$ (which is defined over $k$) of 
the linear system $|H^0(\P^N_k, \sI(d))|$ has the property that
$H \cap W$ is irreducible and smooth for every $W \in \sW$. 
In particular, it is integral.
If $k$ is finite, we can use \cite[Prop.~5.2, Thm.~8.4]{GK} to find a
hypersurface $H \in |H^0(\P^N_k, \sI(d))|(k)$ for all $d \gg 0$ such that
$H \cap W$ is irreducible and smooth for every $W \in \sW$. 
In particular, it is integral.

We choose one hypersurface $H$ as above and let $W' = X' \cap H$. Then
it is clear that
$W'$ is a normal crossing variety of dimension $d-1$ with irreducible
components $\{Y' \cap W', X_2 \cap W', \ldots , X_r \cap W'\}$.
We let $V' = Y' \cap W'$. 
By induction, it follows that $\divf(f) \in \sR^{LW}_0(W', D' \cap W')$.
Since $\psi \colon W' \inj X'$ is a complete intersection closed subscheme
such that $W'_{\rm sing} = X'_{\rm sing} \cap W' = D' \cap W'$, there is a 
push-forward map (e.g., see \cite[Lem.~3.10]{BK})
$\psi_* \colon \sZ_0(W', D' \cap W') \to \sZ_0(X',D')$
such that $\psi_*(\sR^{LW}_0(W', D' \cap W')) \subset \sR^{LW}_0(X',D')$.
It follows that $\divf(f) \in \sR^{LW}_0(X',D')$.

Finally, as $\pi \colon X' \to X$ is a finite sequence of point blow-ups away 
from $D = X_{\rm sing}$,
there is a push-forward map $\pi_* \colon \sZ_0(X', D') \to \sZ_0(X,D)$
such that $\pi_*(\sR^{LW}_0(X', D')) \subset \sR^{LW}_0(X,D)$
(see \cite[Lem.~3.18]{BK}). We conclude that
$\divf(f) = \divf(f)_C = \pi_*(\divf(f)_{C'}) \in 
\sR^{LW}_0(X,D)$. This concludes the proof.
\end{proof}

\subsection{Proof of \thmref{thm:Main-6}}\label{sec:Final}
Let $k$ be an infinite perfect field 
and $X$ a projective normal crossing variety of dimension $d \ge 0$ over $k$.
We write $D$ for $X_\sing$.

We shall now prove our final result which identifies the
motivic cohomology in the bi-degree $(2d,d)$ of $X$ to its Levine-Weibel Chow 
group. One consequence of this is that it allows us to deduce that the
main result of \cite{BKres}, which proved a restriction isomorphism for
relative 0-cycles in a projective and flat family over a Henselian 
discrete valuation
ring, recovers the earlier result of \cite{EKW} when the family is semi-stable,
i.e., the reduced special fiber is a simple normal crossing divisor.
This question was then raised by some of the authors of \cite{EKW}.
The second consequence is that it allows one to describe the
motivic cohomology of a normal crossing variety in terms of 0-cycles 
supported on $X_\reg$ and only one type of relations, unlike in the
description given in \cite{EKW} which uses two types of relations
(see below).

We shall need one more ingredient: a new Chow group of 0-cycles
introduced in \cite[\S~2]{EKW}. 
Let $\sR^{EKW}_0(X,D) \subset \sZ_0(X,D)$ be the subgroup generated
by $\divf(f)$, where $f \in k(C)^{\times}$ is a rational function on a curve 
$C \subset X$ such that the pair $(C,f)$ satisfies either of the conditions 
(1) and (2) below.
\begin{enumerate}
\item
$C$ is an integral curve not contained in $D$ with normalization
$\nu: C^N \to C \inj X$ and $f \in \sO^{\times}_{C^N, \nu^*(D)}$ such that 
$f(x) = 1$ for all $x \in \nu^*(D)$.
\item
$C \subset X$ is a snc subcurve and 
$f \in \sO^{\times}_{C, (C\cap D)}$.
\end{enumerate}

We let $\CH^{EKW}_0(X)$ denote the quotient of $\sZ_0(X,D)$ by the subgroup
$\sR^{EKW}_0(X,D)$. This group was defined in \cite[\S~2]{EKW}, where
it was denoted by $\sC(X)$. One of the main computational results of 
\cite{EKW} is the comparison between the group $\CH^{EKW}_0(X)[\tfrac{1}{p}]$ and 
the cdh motivic cohomology in bi-degree $(2d,d)$ of the normal crossing variety 
$X$, where $p$ is the exponential characteristic of $k$, see \cite[Thm.~7.1]{EKW}. 
There is a canonical surjection
\begin{equation}\label{eqn:surjSNC-EKW}
    \CH_0^{snc}(X) \surj \CH_0^{EKW}(X)
\end{equation}
since both groups are quotients of $\sZ_0(X,D)$.

\vskip .2cm

We now prove \thmref{thm:Main-6}. We restate it for convenience.

\begin{thm}\label{thm:EKW-LW}
  Assume that the condition ~\eqref{eqn:star} holds.
  Then the map $\lambda_X \colon \CH^{LW}_0(X)_\Lambda \to H^{2d}(X, \Lambda(d))$
  is an isomorphism in the following cases.
  \begin{enumerate}
    \item
      ${\rm char}(k) = p > 0$ and $\Lambda =
      \Z[\tfrac{1}{p}]$.
     \item
     $m \in k^{\times}$ and $\Lambda = {\Z}/m$. 
    \end{enumerate}
  \end{thm}
  \begin{proof}
It is shown in \cite[Thm~7.1]{EKW} that the maps 
\begin{equation}\label{eqn:surjSNC-EKW-0}
\Lambda \xrightarrow{\simeq} H^0(k(x), \Lambda(0)) \to 
H^{2d}_c(X \setminus D, \Lambda(d)) \to H^{2d}(X, \Lambda(d))
\end{equation}
for $x \in X \setminus D$ induce a map
$\gamma_X \colon \CH^{EKW}_0(X)_\Lambda \xrightarrow{\simeq} 
H^{2d}(X, \Lambda(d))$. Moreover, this is an isomorphism. 
We remark here that in op. cit., Definition~\ref{defn:ncs} is used for
normal crossing varieties which is a stronger notion than the one we use.
However, the reader can easily check that the proof of \cite[Thm~7.1]{EKW} 
only uses our weaker notion. It suffices now to show that 
there is a canonical surjection  $\psi_X\colon \CH^{EKW}_0(X)_\Lambda \surj 
\CH^{LW}_0(X)_\Lambda$ such that $\gamma_X = \lambda_X \circ \psi_X$.

To show that the identity map of $\sZ_0(X, D)_\Lambda$ descends to $\psi_X$,
we let $C \subset X$ be a curve and $f \in k(C)^{\times}$ a rational function.
We first observe that if the pair $(C,f)$ is of type (1) above, then
$C$ must be contained in
one and only one irreducible component $Y$ of $X$. Moreover, for
this component $Y$, the intersection $E = D \cap Y$ must be a simple normal
crossing divisor on $Y$.  We now conclude easily from Theorems~\ref{thm:Main-2-fin}
and ~\ref{thm:Reg-mod-LW} that $\psi_{X}(\divf(f)) \in \sR^{LW}_0(X,D)_\Lambda$.
If $(C,f)$ is of type (2) above, then we get 
$\psi_X(\divf(f)) \in \sR^{LW}_0(X,D)_\Lambda$
by \lemref{lem:snc-curve-reg}. We have thus shown that $\psi_X \colon
\CH^{EKW}_0(X)_\Lambda \surj \CH^{LW}_0(X)_\Lambda$ is defined.
It is clear from ~\eqref{eqn:surjSNC-EKW-0} that
$\gamma_X = \lambda_X \circ \psi_X$.
This concludes the proof.
\end{proof}

\vskip .4cm

\noindent\emph{Acknowledgements.}
Parts of this project were worked out during authors' visits to 
HIM at Bonn, Universit{\"a}t Regensburg and University of Tokyo
during past years. The authors thank these institutes for supporting their
visits. The authors are indebted to the anonymous referees for reading the
manuscript thoroughly and suggesting many improvements.

\end{document}